\documentclass[12pt,twoside]{amsart}
\usepackage{amssymb}

\usepackage[all]{xy}
\nonstopmode

\numberwithin{equation}{section} 

\font\tengothic=eufm10 scaled\magstep 1 \font\sevengothic=eufm7
scaled\magstep 1
\newfam\gothicfam
       \textfont\gothicfam=\tengothic
       \scriptfont\gothicfam=\sevengothic
\def\goth#1{{\fam\gothicfam #1}}

\newtheorem{theorem}{Theorem}[section]
\newtheorem{lemma}[theorem]{Lemma}
\newtheorem{proposition}[theorem]{Proposition}
\newtheorem{corollary}[theorem]{Corollary}
\newtheorem{conjecture}[theorem]{Conjecture}

\theoremstyle{definition}
\newtheorem{definition}[theorem]{Definition} % \theoremstyle{remark}
\newtheorem{remark}[theorem]{Remark}
\newtheorem{example}[theorem]{Example}

\newcommand{\codim}{\operatorname{codim}}
\newcommand{\coker}{\operatorname{coker}}

\newcommand{\Hom}{\operatorname{Hom}}

\newcommand{\ext}{\operatorname{ext}}
\newcommand{\Ext}{\operatorname{Ext}}

\newcommand{\depth}{\operatorname{depth}}

\newcommand{\im}{\operatorname{im}}

\newcommand{\Hi}{\operatorname{Hilb}}
\DeclareMathOperator{\GradAlg}{GradAlg}
\newcommand{\smallbox}{\hbox{\tiny $\qed$}}
\newcommand{\Proj}{\operatorname{Proj}}
\newcommand{\Spec}{\operatorname{Spec}}

\newcommand{\cA}{{\mathcal A}}
\newcommand{\cB}{{\mathcal B}}
\newcommand{\cH}{{\mathcal H}}
\newcommand{\cI}{{\mathcal I}}

\newcommand{\cC}{{\mathcal C}}

\newcommand{\cF}{{\mathcal F}}
\newcommand{\cG}{{\mathcal G}}
\newcommand{\cO}{{\mathcal O}}

\newcommand{\cN}{{\mathcal N}}

\newcommand {\PP}{\mathbb{P}}

\newcommand {\ra}{\longrightarrow}

\begin{document}
\title[]{Families of low dimensional determinantal schemes.}

\author[Jan O.\ Kleppe]{Jan O.\ Kleppe} %, Rosa M.\ Mir\'o-Roig$^{*}$}
%\author[]{}
\address{Faculty of Engineering, Oslo University College,
         Pb. 4 St. Olavs plass, N-0130 Oslo, Norway}
\email{JanOddvar.Kleppe@iu.hio.no}
%\address{Facultat de Matem\`atiques,
%Departament d'Algebra i Geometria, Gran Via de les Corts Catalanes
%585, 08007 Barcelona, SPAIN } \email{miro@ub.edu}

\date{\today}
%\thanks{$^*$ Partially supported by BFM2001-3584.}


%%%%%%%%%%%%%%%%%%%%%%%%%%%%%%%%

\begin{abstract} 
  A scheme $X\subset \PP^{n}$ of codimension $c$ is called {\it standard
    determinantal} if its homogeneous saturated ideal can be generated by the
  $t \times t$ minors of a homogeneous $t \times (t+c-1)$ matrix $(f_{ij})$.
  Given integers $a_0\le a_1\le ...\le a_{t+c-2}$ and $b_1\le ...\le b_t$, we
  denote by $W_s(\underline{b};\underline{a}) \subset {\rm Hilb}(\PP^{n})$ the
  stratum of standard determinantal schemes where $f_{ij}$ are homogeneous
  polynomials of degrees $a_j-b_i$ and ${\rm Hilb}(\PP^{n})$ is the Hilbert
  scheme (if $n-c > 0$, resp. the postulation
  Hilbert  scheme if $n-c = 0$). \\[2mm]
  Focusing mainly on zero and one dimensional determinantal schemes we
  determine the codimension of $W_s(\underline{b};\underline{a})$ in $ {\rm
    Hilb}(\PP^{n})$ and we show that $ {\rm Hilb}(\PP^{n})$ is generically
  smooth along $W_s(\underline{b};\underline{a})$ under certain conditions.
  For zero dimensional schemes (only) we find a counterexample to the
  conjectured value of $\dim W_s(\underline{b};\underline{a})$ appearing in
  \cite{KM}.
\end{abstract}

%%%%%%%%%%%%%%%%%%%%%%%%%%%%%%%

\maketitle

%\tableofcontents

%%%%%%%%%%%%%%%%%%%%%%%%%%%%%%%%%%%%%%%%%%%%%%%%%

\section{Introduction} \label{intro} The goal of this paper is to study
maximal families of determinantal schemes. Recall that a scheme $X\subset
\PP^{n}$ of codimension $c$ is called determinantal if its homogeneous
saturated ideal can be generated by the $r \times r$ minors of a homogeneous
$p \times q$ matrix $(f_{ij})$ with $c=(p-r+1)(q-r+1)$. If $r={\rm
  \min}(p,q)$, then $X$ is called {\it standard determinantal}. $X$ is
called %said to be
{\em good determinantal} if it is standard determinantal and a generic
complete intersection.

Let $ \Hi (\PP^{n})$ be the Hilbert scheme (resp. postulation Hilbert scheme,
i.e. the Hilbert scheme of constant Hilbert function) parameterizing closed
subschemes of $\PP^{n}$ of dimension $n-c > 0$ (resp. $n-c= 0$). Given
integers $a_1\le a_2\le ...\le a_{p}$ and $b_1\le ...\le b_q$, we denote by
$W(\underline{b};\underline{a})$ (resp. $W_s(\underline{b};\underline{a}))$
the stratum in ${\rm Hilb}(\PP^{n})$ consisting of good (resp. standard)
determinantal schemes %as above 
where $f_{ij}$ are homogeneous polynomials of
degrees $a_j-b_i$. Then $W_s(\underline{b};\underline{a}) $ is irreducible and
$W(\underline{b};\underline{a})\ne \emptyset $ if and only if
$W_s(\underline{b};\underline{a}) \ne \emptyset$ (Corollary~\ref{WWs}).
 \vskip 5 pt

 In this paper we focus, notably for {\it zero dimensional} schemes, on the
 following problems. \vskip 2 pt (1) Determine when the closure of
 $W(\underline{b};\underline{a})$ is an irreducible component of $ {\rm
   Hilb}(\PP^{n})$. \vskip 1 pt (2) Find the codimension of
 $W(\underline{b};\underline{a})$ in $ {\rm Hilb}(\PP^{n})$ if its closure is
 not a component. \vskip 1 pt (3) Determine when $ {\rm Hilb}(\PP^{n})$ is
 generically smooth along $W(\underline{b};\underline{a})$. \vskip 5 pt This
 paper generalizes and completes several results of \cite{KMMNP} and \cite{KM}
 for schemes of dimension 0 or 1. Moreover we announced in \cite{KM}, Rem.\!
 6.3 that \cite{KMMNP}, \S 10 contains inaccurate results in the
 zero dimensional case, which we fully correct in this paper
 (Remark~\ref{correctKM}).

 By successively deleting columns of the matrix associated to a determinantal
 scheme $X$, we get a nest (``flag'') of closed subschemes $X =X_c \subset
 X_{c-1} \subset ... \subset X_2 \subset {\PP^{n} }$. We prove our results
 inductively by considering the smoothness of the Hilbert flag scheme of pairs
 and its natural projections into the Hilbert schemes. Note that, for $c=2$,
 one knows that the closure $\overline {W(\underline{b};\underline{a})}$ is a
 generically smooth irreducible component of $ {\rm Hilb}(\PP^{n})$ (i.e. $
 {\rm Hilb}(\PP^{n})$ is smooth along some non-empty {\it open subset $U$ of}
 $ {\rm Hilb}(\PP^{n})$ satisfying $ U \subset
 W(\underline{b};\underline{a})$), see Theorem~\ref{MainThm2}.

 \vskip 2 pt In this approach we need to prove that certain (kernels of) ${\rm
   Ext^1}$-groups vanish or to compute its dimensions. If $\dim X =1$ (resp.
 0), then one (resp. 2 or 3) of these ${\rm Ext^1}$-groups may be non-zero and
 its dimension (resp. the sum of its dimensions) is precisely the codimension
 of $W(\underline{b};\underline{a})$ in $ {\rm Hilb}(\PP^{n})$ under certain
 assumptions, see  Theorem~\ref{Dim0Thm},  Proposition~\ref{Dim0Prop} and
 Proposition~\ref{Dim1Prop} of Section 4. These are main results of this
 paper, together with the key Proposition~\ref{mainTechn} which through
 Proposition~\ref{varMainTechn} and Lemma~\ref{unobst} give the tools we need
 in the proofs. As a consequence, if the mentioned ${\rm Ext^1}$-groups vanish
 and $c \le 5$ or 6, we get that the closure $\overline{
   W(\underline{b};\underline{a})}$ is a generically smooth irreducible
 component of $\Hi (\PP^{n})$ and that every deformation of a general $X$ of
 $W(\underline{b};\underline{a})$ comes from deforming the defining matrix
 $(f_{ij})$ of $X$. Note that this conclusion holds if $\dim X \ge 2$ and $3
 \le c \le 4$ because the above ${\rm Ext^1}$-groups vanish by \cite{KM} and
 \cite{KMMNP}. If the codimension of $W(\underline{b};\underline{a})$ in $
 {\rm Hilb}(\PP^{n})$ is positive, there are deformations of $X$ which do not
 come from deforming the matrix $(f_{ij})$. In the proofs we %systematically
 use results of \cite{KM} and \cite{KMMNP} (see Section 3 which also contains
 a counterexample to the Conjectures of \cite{KM} in the case $\dim X = 0$),
 as well as the Eagon-Northcott and Buchsbaum-Rim complexes
 (\cite{e-n},\cite{BR}, \cite{eise}). We give many examples, supported by
 Macaulay 2 computations \cite{Mac}, to illustrate the results.

 As an application we expect that the results for zero dimensional schemes $X =
 \Proj(A)$ of this paper, together with the main result of \cite{K04} in which
 artinian Gorenstein rings are obtained by dividing $A$ with ideals being
 isomorphic to a fixed twist of the canonical module of $A$, can be used in
 the classification of Gorenstein quotients of a polynomial ring of e.g.
 codimension 4 from the point of view of determining ${\rm PGor}(H)$, cf.
 \cite{IK}, \cite{K98}.
% and the deformations of a
% regular section of certain Cohen-Macaulay modules of ranks 1. 
%Moreover see also \cite{e-h}, \cite{b-v}, \cite{elli}, \cite{BH} and
%\cite{eise} for background and a list of important papers.

 Some of the results of this paper were lectured at the "4th World Conference
 on 21st Century Mathematics 2009'' in Lahore in March 2009. The author thanks
 the organizers for their hospitality. Moreover I thank prof. R.M. Mir\'o-Roig
 at Barcelona for interesting comments and our discussion on the
 Conjectures~\ref{conj1} and ~\ref{conj2} and the
 counterexample~\ref{counter}.

\vskip 4mm

{\bf Notation:}  In this paper $\PP:=\PP^n$ will be the projective $n$-space
over an algebraically closed field $k$, $R=k[x_0, x_1, \dots ,x_n]$ is a
polynomial ring and $\goth m= (x_0, \dots ,x_n)$. %The sheafification of a
% graded $R$-module $M$ will be denoted by $\tilde{M}$ and the support of $M$
% by $Supp(M)$.

We mainly keep the notations of \cite{KM}. If $X \subset Y$ are closed
subschemes of $\PP^n$, we denote by ${\mathcal I}_{X/Y}$ (resp. ${\mathcal
  N}_{X/Y}$) the ideal (resp. normal) sheaf of $X$ in $ Y$. Note that by the
codimension, ${\rm codim}_Y X$, of $X$ in $Y$ we simply mean $\dim Y -\dim X$,
also in non arithmetically Cohen-Macaulay cases. For any closed subscheme $X$
of $\PP^n$ of codimension $c$, we denote by ${\mathcal I}_X$ its ideal sheaf,
${\mathcal N}_X$ its normal sheaf, $ I_X=H^0_{*}({\mathcal I}_X)$ its
saturated homogeneous ideal and we let $\omega_X ={\mathcal E}xt^c_{{\mathcal
    O}_{\PP^n}} ({\mathcal O}_X,{\mathcal O}_{\PP^n})(-n-1)$. When {\it we
  write} $X=\Proj(A)$ {\it we take} $A:=R/I_X$ and $K_A=\Ext^c_R (A,R)(-n-1)$.
We denote the group of morphisms between coherent
$\cO_X$-modules % $\cF$ to $\cG$
by $\Hom_{\cO_X}(\cF,\cG)$ while $\cH om_{\cO_X}(\cF,\cG)$ denotes the sheaf
of local morphisms. % of $\cF$ into $\cG$. We often omit
% $\cO_X$ in $\Hom_{\cO_X}(\cF,\cG)$ (resp. $\cH om_{\cO_X}(\cF,\cG)$) if the
% underlying scheme $X$ is evident.
Moreover we set $\hom(\cF,\cG)=\dim_k\Hom(\cF,\cG)$ and we correspondingly use
small letters for the dimension, as a $k$-vector space, of similar groups. %For
%any quotient $A$ of $R$ of codimension $c$, we let
%$I_A=\ker(R\twoheadrightarrow A)$ , $N_A=\Hom_R(I_A,A)$ 
%and $K_A=\Ext^c_R
%(A,R)(-n-1)$. When we write $X=\Proj(A)$, we let $A=R/I_X$. % and $K_X=K_A$.

We denote the Hilbert scheme by $\Hi ^p(\PP^n)$, $p$ % \in \QQ[s]$
the Hilbert polynomial \cite{G}, and $(X) \in \Hi ^p(\PP^n)$ for the point
which corresponds to the subscheme $X\subset \PP^n$. We denote by
$\GradAlg(H)$, or $\Hi ^H(\PP^n)$, the representing object of the functor
which parameterizes flat families of graded quotients $A$ of $R$ of $\depth A
\ge \min(1,\dim A)$ and with Hilbert function $H$ (\cite{K98}, \cite{K04}),
and we call it ``the postulation Hilbert scheme'' (\cite{K07}, \S 1.1) even
though it may be different from the parameter space studied by Gotzmann,
Iarrobino and others (\cite{Go}, \cite{IK}) who study the ``same'' scheme with
the reduced scheme structure (ours may be non-reduced and is equivalent to the
Hilbert scheme of constant postulation considered in \cite{MP} in the curve
case. They are both special cases of the multigraded Hilbert scheme of Haiman
and Sturmfels \cite{HS}). Again we let $(A)$, or $(X)$ where $X= \Proj(A)$,
denote the point of $\GradAlg(H) $ which corresponds to $A$. Note that if
$\depth_{\goth m}A \geq 1$ and $\ _0\!\Hom_R (I_X,H^{1}_{\goth m}(A)) = 0$,
then
\begin{equation}  \label{Grad}
  \GradAlg(H) \simeq \Hi ^p(\PP^n) \ \ \ {\rm at} \ \ \ (X) \   ,
\end{equation} 
and hence we have an isomorphism $ _0\!\Hom (I_{X},A) \simeq \ H^0({\mathcal
  N}_X)$ of their tangent spaces (cf. \cite{elli} for the case $\depth_{\goth
  m}A \geq 2$, and \cite{K04}, (9) for the general case). If \eqref{Grad}
holds and $X$ % \hookrightarrow \PP^n$ 
is generically a complete intersection, then $
_0\!\Ext^1_A(I_{X}/I_{X}^2,A)$ is an obstruction space of $ \GradAlg(H)$ and
hence of $ \Hi ^p(\PP^n)$ at $(X)$ (\cite{K04}, \S 1.1). When we simply write
$\Hi (\PP^n)$, we interpret it as the Hilbert scheme (resp. postulation
Hilbert scheme) if $ n-c >0$ (resp. $ n-c =0$). By definition $X$ (resp. $A$)
is {\it unobstructed} if $\Hi ^p(\PP^n)$ (resp. $\Hi ^H(\PP^n)$) is smooth at
$(X)$. Note that we called $X$ {\it H-unobstructed} in \cite{KMMNP} if $A$ was
unobstructed.

%
%The pullback of the universal family on $ \Hi ^p(\PP^n)$ via a morphism
%$\psi:W\ra \Hi ^p(\PP^n)$ yields a flat family over $W$, and we will write
%$(X)\in W$ for a member of that family as well. 
%By definition if a {\it general} $X$ of some subscheme $ W \subset \Hi(\PP^n)$
%has a certain property then there is a non-empty open dense subset $U$ of $W$
%such that all members of $U$ have this property. Moreover, 
We say that $X$ is {\it general} in some irreducible subset $ W \subset
\Hi(\PP^n)$ if $(X)$ belongs to a sufficiently small open subset $U$ of $W$
(so any $(X)$ in $U$ has all the openness properties that we want to require).
 %it will be clear from
%the conteit is a special case of the multigraded
%Hilbert scheme of Haiman and Sturmfels). 
%%%%%%%%%%%%%%%%%%%%%%%%%%%%%%%%%%%%%%%%%%%%%%%%%%%%%%%%%%%%%%%%%%%%%%%%%%%%%

\section{Background}

In this section we recall some basic results on standard (resp. good)
determinantal schemes needed in the sequel, see \cite{b-v}, \cite{eise} and
\cite{BH} for more details. Let
\begin{equation}\label{gradedmorfismo} \varphi:F=\bigoplus
  _{i=1}^tR(b_i)\longrightarrow G:=\bigoplus_{j=0}^{t+c-2}R(a_j)
\end{equation}
be a graded morphism of free $R$-modules and let
$\cA=(f_{ij})_{i=1,...t}^{j=0,...,t+c-2}$, $\deg f_{ij}=a_j-b_{i}$, be a
$t\times (t+c-1)$ homogeneous matrix which represents the dual
$\varphi^*:=\Hom_R(\varphi,R)$. Let $I(\cA)=I_t(\cA)$ (or $I_t(\varphi )$) be
the ideal of $R$ generated by the maximal minors of $\cA$. In the following we
always suppose 
$$ c\ge 2,\ \ t\ge 2, \ \ b_1 \le ... \le b_t \ \ \  {\rm and} \ \ \ a_0 \le a_1\le ...
\le a_{t+c-2}.$$ Recall that a codimension $c$ subscheme $X\subset \PP^{n}$ is
standard determinantal if $I_X=I(\cA)$ for some homogeneous $t\times (t+c-1)$
matrix $\cA$ as
above. %in which case we have $a_{i-1}-b_i> 0$ for $i=1,...,t$.
Moreover $X\subset \PP^{n}$ is a \emph{good determinantal} scheme if
additionally, $\cA$ contains a $(t-1)\times (t+c-1)$ submatrix (allowing a
change of basis if necessary) whose ideal of maximal minors defines a scheme
of codimension $c+1$. Note that if $X$ is standard determinantal and a generic
complete intersection in $ \PP^{n}$, then $X$ is \emph{good determinantal},
and conversely \cite{KMNP}, Thm.\! 3.4. Without loss of generality we assume
that $\cA$ is minimal; i.e., $f_{ij}=0$ for all $i,j$ with $b_{i}=a_{j}$.

Let $W(\underline{b};\underline{a})$ (resp. $W_s(\underline{b};\underline{a}))$
be the stratum in $ {\rm Hilb}(\PP^{n})$ consisting of good (resp. standard)
determinantal schemes. % as above.  $W_s(\underline{b};\underline{a})$ is
%irreducible (see \cite{KM}, p. 2877), and 
% where $f_{ij}$ are homogeneous polynomials of degrees $a_j-b_i$.
By \cite{KM}, Cor.\! 2.5 and see the end of p.\! 2877, we get

\begin{corollary} \label{WWs} $W(\underline{b};\underline{a})\ne \emptyset $
  if and only if $W_s(\underline{b};\underline{a}) \ne \emptyset $ if and only
  if $a_{i-1}-b_i> 0$ for $i=1,...,t$. Moreover, their closures in $ \Hi
  (\PP^n)$ % (resp. in $ \Hi ^H(\PP^n)$ if $n=c$)
  are equal and irreducible.
\end{corollary}

\vskip 2mm Let $A:=R/I_t(\cA)$ %be standard determinantal and let
and $M:= \coker (\varphi^*)$. Using the generalized Koszul complexes
associated to a codimension $c$ standard determinantal scheme $X$, one knows
that the {\em Eagon-Northcott complex} yields the following minimal free
resolution
\begin{equation}\label{EN}0 \ra \wedge^{t+c-1}G^* \otimes S_{c-1}(F)\otimes
  \wedge^tF\ra \wedge^{t+c-2} G ^*\otimes S _{c-2}(F)\otimes \wedge ^tF\ra
  \ldots  \end{equation}  
$$ \ra
\wedge^{t}G^* \otimes S_{0}(F)\otimes \wedge^tF\ra R \ra A \ra 0 $$ of $A$ and
that the {\em Buchsbaum-Rim complex} yields a minimal free resolution
of $M$;
\begin{equation}\label{BR}
0 \ra \wedge^{t+c-1}G^* \otimes S_{c-2}(F)\otimes \wedge^tF\ra
\wedge^{t+c-2} G ^*\otimes S _{c-3}(F)\otimes \wedge ^tF\ra \ldots 
\end{equation}
 $$ \ra \wedge^{t+1}G^* \otimes S_{0}(F)\otimes \wedge^tF \ra G^*\ra \ F^* \ra
 M \ra 0 .$$

 See, for instance \cite{b-v}; Thm.\! 2.20 and \cite{eise}; Cor.\! A2.12 and
 Cor.\! A2.13. Note that \eqref{EN} show that any standard determinantal
 scheme is arithmetically Cohen-Macaulay (ACM). % $X\subset \PP^{n}$ is ACM.
% Moreover, in codimension 2, the converse
%is true: If $X\subset \PP^{n+2}$ is an ACM, closed subscheme of
%codimension 2 then it is standard determinantal (Hilbert-Burch
%Theorem).}

Let ${\cB}$ be the matrix obtained deleting the last column of ${\cA}$ and let
$B$ be the $k$-algebra given by the maximal minors of ${\cB}$. Let
$Y=\Proj(B)$. The transpose of ${\cB}$ induces a map $ \phi:F=\oplus
_{i=1}^tR(b_i)\rightarrow G':=\oplus _{j=0}^{t+c-3}R(a_j)$. Let $M_{\cB}$ be
the cokernel of $\phi^*=\Hom_R(\phi,R)$ and let $M_{\cA}=M$.
%\begin{equation}\label{defMi}
%G_i^*\stackrel {\varphi_i^*}{ \longrightarrow} F^* \longrightarrow
%M_i\simeq  \Ext ^1_R(B_i,R)\longrightarrow 0
%\end{equation} % (i.e.
%$R\twoheadrightarrow D_2 \twoheadrightarrow D_3 \twoheadrightarrow
%...\twoheadrightarrow D_c=A$), then $M_i$ is a $D_i$-module and
In this situation we recall that there is an exact sequence
\begin{equation}\label{Mi}
0\longrightarrow B \longrightarrow M_{\cB}(a_{t+c-2})
\longrightarrow M_{\cA}(a_{t+c-2}) \longrightarrow 0
\end{equation}
in which $B \longrightarrow M_{\cB}(a_{t+c-2})$ is a regular section given by
the last column of ${\cA}$. Moreover,
\begin{equation}\label{Di}
0\longrightarrow
M_{\cB}(a_{t+c-2})^* :=\Hom_{B}(M_{\cB}(a_{t+c-2}),B)\longrightarrow B
\longrightarrow A \longrightarrow 0
\end{equation}
is exact by \cite{KMNP} or \cite{KMMNP}, (3.1), i.e. we may put
$I_{X/Y}:=M_{\cB}(a_{t+c-2})^*$. Due to \eqref{BR}, $M$ is a maximal
Cohen-Macaulay $A$-module, and so is $I_{X/Y}$ by (\ref{Di}). By \cite{eise}
we have $K_{A}(n+1)\simeq S_{c-1}M_{\cA}(\ell_c)$, and hence $K_{B}(n+1)\simeq
S_{c-2}M_{\cB}(\ell_{c-1})$, where
\begin{equation}\label{ell}
 \ell_i :=\sum_{j=0}^{t+i-2}a_j-\sum_{k=1}^tb_k \ \ {\rm for} \ \ 2 \le i \le c.
\end{equation}
Hence by successively deleting columns from the right hand side of $\cA$, and
taking maximal minors, one gets a flag of determinantal subschemes
\begin{equation}\label{flag}
  ({\mathbf X.}) :  X = X_c \subset X_{c-1} \subset ...  \subset X_{2}
  \subset  \PP^{n}  
\end{equation}
where each $X_{i+1} \subset X_i$ (with ideal sheaf ${ \mathcal I_{X_{i+1}/X_i}}
= {\mathcal I_i}$) is of codimension 1, $X_{i}
\subset \PP^{n}$ is of codimension $i$  and where there exist
${\mathcal O}_{X_i}$-modules ${\mathcal M_i}$ fitting into short exact
sequences
\begin{equation}\label{fix-ref}
0\rightarrow {\mathcal O}_{X_i}(-a_{t+i-1})\rightarrow {\mathcal
M}_i \rightarrow {\mathcal M}_{i+1} \rightarrow 0 \ \  {\rm for} \
\ 2 \leq i \leq c-1,
\end{equation}
such that ${\mathcal I}_i (a_{t+i-1})$ is the ${\mathcal O}_{X_i}$-dual of
${\mathcal M}_i$ for $2 \leq i \leq c$, and ${\mathcal M}_{2}$ is a twist of
the canonical module of $X_2$. In this context we let $D_i:= R/I_{X_i}$,
$I_{D_i} = I_{X_i}$ and $I_i:=I_{D_{i+1}/I_{D_i}}$.

\begin{remark} \label{dep} %Assume $b_1\le ... \le b_t$ and
%$a_0\le a_1 \le ... \le a_{t+c-2}$. 
  Let $\alpha$ be a positive integer. If $X$ is general in
  $W(\underline{b};\underline{a})$ and $u_{i-\min (\alpha ,t),i}=a_{i-\min
    (\alpha ,t)}-b_i\ge 0$ for $\min (\alpha ,t)\le i\le t$, then $X_{j}$, for
  all $j=2, \cdots , c$, is non-singular except for a subset of codimension at
  least $ \min \{2\alpha -1, j+2 \}$, i.e.
\begin{equation}\label{a-b} \codim_{X_j}Sing(X_j)\ge \min\{2\alpha
  -1, j+2\} . \end{equation} As observed in Rem.\! 2.7 of \cite{KM}, this
follows from the Theorem of  
\cite{chang} by
arguing as in \cite{chang}, Example 2.1. See  \cite{Sa} for a special case. In
particular, if 
$\alpha \ge 3$, we get that for each $i>0$, the closed embeddings
$X_i\subset \PP^{n}$ and $X_{i+1}\subset X_{i}$ are local
complete intersections outside some set $Z_i$ of codimension at
least $ \min(4,i+1)$ in $X_{i+1}$ ($\depth_{Z_i}\cO_{X_{i+1}}\ge
\mbox{min}(4,i+1)$), cf. next paragraph.
% Moreover, taking $\alpha =1$, we deduce from (\ref{a-b}) and
% Corollary~\ref{WWs} a non-empty $W(\underline{b};\underline{a})$ always
% contains a reduced determinantal scheme.
\end{remark}
\vskip 2mm In what follows we always let $Z \subset X$ (and similarly for
$Z_i\subset X_i$) be some closed subset such that $U:=X-Z\hookrightarrow
\PP^{n}$ (resp. $U_{i}:=X_{i}-Z_{i}\hookrightarrow \PP^{n}$) is a local
complete intersection (l.c.i.). Since the 1. Fitting ideal of $M$ is equal to
$I_{t-1}(\varphi)$, we get that $\tilde{M}$ is locally free of rank one
precisely on $X-V(I_{t-1}(\varphi))$ \cite{BH}, Lem.\! 1.4.8. Since the set of
non locally complete intersection points of $X\hookrightarrow \PP^{n}$ is
exactly $V(I_{t-1}(\varphi))$ by e.g. \cite{ulr}, Lem.\! 1.8, we get that
$U\subset X-V(I_{t-1}(\varphi))$ and that $\tilde{M}$ is locally free on $U$.
Indeed ${\mathcal M}_i$ and $\cI _{X_{i}}/\cI ^2_{X_{i}}$ are locally free on
$U_i$, as well as on $U_{i-1} \cap X_i$. Note that since $V(I_{t-1}(\cB))
\subset V(I_{t}(\cA))$, we may suppose $Z_{i-1} \subset X_i$!

Finally notice that there is a close relation between $ M(a_{t+c-2})$ and the
normal module $N_{X/Y}$ of the quotient $B \simeq R/I_Y \rightarrow A \simeq
R/I_X$. If we suppose $\depth_{I(Z)}B\ge 2$ where now $Y-Z\hookrightarrow
\PP^{n}$ is an l.c.i., we get by applying $ \Hom_{B}(I_{X/Y},-)$ to
(\ref{Di}), that
\begin{equation} \label{DiMi}
0\longrightarrow B \longrightarrow M_{\cB}(a_{t+c-2})
\longrightarrow N_{X/Y}
\end{equation}
is exact. Hence we have an injection $M_{\cA}(a_{t+c-2})\hookrightarrow
N_{X/Y}$, which in the case $\depth_{I(Z)}B\ge 3$ leads to an isomorphism
$M_{\cA}(a_{t+c-2})\simeq N_{X/Y}$. Indeed, this follows from the more general
fact (by letting $L=N=I_{X/Y}$) that if $L$ and $N$ are finitely generated
$B$-modules such that $\depth_{I(Z)}L\ge r+1$ and $\tilde{N}$ is locally free
on $U:=Y-Z$, then the natural map
\begin{equation} \label{NM}
\Ext^{i}_B(N,L)\longrightarrow
H_{*}^{i}(U,{\cH}om_{{\cO}_Y}(\tilde{N},\tilde{L}))
\end{equation}
is an isomorphism (resp. an injection) for $i<r$ (resp. $i=r$), and
$H_{*}^{i}(U,{\cH}om_{{\cO}_Y}(\tilde{N},\tilde{L}))\simeq
H^{i+1}_{I(Z)}(\Hom_B(N,L))$ for $i>0$, cf.\! \cite{SGA2}, exp.\! VI. Note
that we interpret $I(Z)$ as $\goth m$ if $Z= \emptyset$.

%%%%%%%%%%%%%%%%%%%%%%%%%%%%%%%%%%%%%%%%%%%%%%%%%%%%%%%%%%%%%%%%%%%%%%%%%%

\section{The dimension of the determinantal locus}

In \cite{KM} we conjectured the dimension of $ W(\underline{b};\underline{a})$
in terms of the invariant {\small
\begin{equation} \label{lamda} \lambda_c:= \sum_{i,j}
    \binom{a_i-b_j+n}{n} + \sum_{i,j} \binom{b_j-a_i+n}{n} - \sum _{i,j}
    \binom{a_i-a_j+n}{n}- \sum _{i,j} \binom{b_i-b_j+n}{n} + 1. 
  \end{equation} }
Here the indices belonging to $a_j$ (resp. $b_i$) range over $0\le j \le
t+c-2$ (resp. $1\le i \le t$), $ \binom{a}{n}= 0$ if $a < n$ and we always
suppose  $ 
W(\underline{b};\underline{a}) \ne \emptyset$ in the following, cf.
Corollary~\ref{WWs}.
% for $ W(\underline{b};\underline{a})= \emptyset$. We define $\lambda _{c-1}$
% by the analogous expression where now $a_j$ (resp $b_i$) ranges over $0\le j
% \le t+c-3$ (resp. $1\le i \le
% t$). %It follows after a straightforward computation that
% {\small
%\begin{equation}\label{lambdas}
%\lambda _c=\lambda_{c-1}+\sum_{i=1}^t
%\binom{a_{t+c-2}-b_i+n}{n}- \sum _{j=0}^{t+c-3}
%\binom{a_{t+c-2}-a_j+n}{n}-
% \sum _{j=0}^{t+c-2} \binom{a_j-a_{t+c-2}+n}{n}.
%\end{equation} }

\begin{conjecture} \label{conj1}
  Given integers $a_0\le a_1\le ... \le a_{t+c-2}$ and $b_1\le ...\le b_t$, we
  set $\ell_i :=\sum_{j=0}^{t+i-2}a_j-\sum_{k=1}^tb_k$ and $h_{i-3}:=
  2a_{t+i-2}-\ell_i +n$, for $i=3,4,...,c$. %(interpreting $h_{c-3}$ as
  %$h_{c-3}:= 2a_{t+c-2}-\ell +n$). 
  Assume $a_{i-\min ([c/2]+1,t)}\ge b_{i}$ for $\min ([c/2]+1,t)\le i \le t$.
  Then we have
 \[
 \dim W(\underline{b};\underline{a}) = \lambda_c+ K_3 + K_4+...+K_c \]
%\[ \sum_{i=1}^{c-3}\left (\sum
%_{r+s=i \atop r \ge 0, s \ge 0} \sum _{0\le i_1< ...< i_{r}\le t+i \atop 1\le
%j_1\le...\le j_s \le t } (-1)^{i-r} \binom{h_i+a_{i_1}+\cdots
%+a_{i_r}+b_{j_1}+\cdots +b_{j_s} }{n}\right  ).
%\]
where $K_3=\binom{h_0}{n}$ and $K_4= \sum_{j=0}^{t+1} \binom{h_1+a_j}{n}-
\sum_{i=1}^{t} \binom{h_1+b_i}{n}$ and in general \[ K_{i+3}= \sum _{r+s=i
  \atop r , s \ge 0} \sum _{0\le i_1< ...< i_{r}\le t+i \atop 1\le
  j_1\le...\le j_s \le t } (-1)^{i-r} \binom{h_i+a_{i_1}+\cdots
  +a_{i_r}+b_{j_1}+\cdots +b_{j_s} }{n} \  {\rm for}  \ 0 \le i \le c-3. \]
\end{conjecture}

For the special case where all the entries of $\cA $ have the same degree,
this means:

\begin{conjecture} \label{conj2} Let $W(\underline{0};\underline{d})$ be the
  locus of good determinantal schemes in $\PP^{n}$ of codimension $c$ given by
  the maximal minors of a \ $t\times (t+c-1)$ matrix with entries homogeneous
  forms of degree $d$. Then, $$\dim W(\underline{0};\underline{d}) =
  t(t+c-1)\binom{d+n}{n}- t^2-(t+c-1)^2+1.$$
\end{conjecture}

In \cite{KM} we proved that the right hand side in the formula for $\dim
W(\underline{b};\underline{a})$ in the Conjectures is always an upper bound
for $\dim W(\underline{b};\underline{a})$ (\cite{KM}, Thm.\! 3.5), and
moreover, that the Conjectures hold in the range
\begin{equation} \label{2c5}
2 \le c \le 5 \  \
  {\rm and} \ \ n-c > 0 \ {\rm (\, supposing \ } char k = 0 {\rm \ if} \ c =
    5\, )\, ,
  \end{equation}
  as well as for large classes in the range $c \ge 2$ (without assuming $n >
  c$), cf. \cite{KM}, \S 4.
%, cf. \cite{KM}, Cor.\! 4.10, 4.14 and 4.15 and see \cite{elli},
%  \cite{KMMNP} for the case $2 \le c \le 3$ and $ n-c > 0$.
%
  \begin{example}  [Counterexample to the Conjectures in the
    range $n=c\ge 3$] \  \label{counter} 

    Let $\cA$ be a general $2 \times (c+1)$ matrix of linear entries. The
    vanishing of all $2 \times 2$ minors defines a reduced scheme $X$ of $c+1$
    general points in $\PP^{c}$. The conjectured dimension is $c(c+1)+c-2$
    while the dimension of the postulation Hilbert scheme, $\dim_{(X)}\Hi^H
    (\PP^{c})$ is at most $c(c+1)$. Hence
   % Moreover since $\ _0\!\Hom_R(I_X,H^{1}_{\goth m}(A)) = 0$ we may see from
   % \eqref{Grad} that 
$$\dim W({0,0};{1,1,...,1}) \le c(c+1) \ .$$ This contradicts both conjectures
for every $c \ge 3$.
\end{example}
We have, however, looked at many examples in the range $a_0 > b_t$ where we
have used Macaulay 2 to compute necessary invariants (cf. \eqref{maineq}
below), without finding more counterexamples. The counterexample we have is
only for zero dimensional schemes. % For
%positively dimensional schemes we have no counterexamples, and in
%the latter case we strongly expect the conjectures to be true.

Now we recall a few statements from the proof of \eqref{2c5} %and an
% inductive formula for $\dim W(\underline{b};\underline{a})$
and a variation which we will need in the next section. In the proof we used
induction on $c$ by successively deleting columns of the largest possible
degree. Hence we computed the dimension of $W(\underline{b};\underline{a})$,
$\underline{a} = a_0, a_1,..., a_{t+c-2}$ in terms of dimension of
$W(\underline{b};\underline{a'})$, where $\underline{a'} = a_0, a_1,...,
a_{t+c-3}$. As in \S 2, we let %take a good determinantal scheme
$X=\Proj(A)$ belong to $W(\underline{b};\underline{a})$ and we let
$Y=\Proj(B)$, $(Y) \in W(\underline{b};\underline{a'})$, be obtained by
deleting the last column of $\cA$. %To do so, we consider the Hilbert
% flag scheme $D(p,q)$ parameterizing "pairs" $X\subset Y$ of closed
% subschemes of $\PP^{n}$ with Hilbert polynomial $p$ and $q$ respectively.
% Denoting by $$m_i(\nu )=\dim _kM_i(a_{t+i-2})_{\nu}$$ Ww
We have
%
%\begin{proposition}\label{main0} (Prop.\! 4.1 of \cite{KM}) \ Let $c\ge 3$ and
%  suppose $W(\underline{b};\underline{a})\ne \emptyset$ and 
%  $\depth_{I(Z)}B\ge 2$ for a general $Y=\Proj(B)\in
%  W(\underline{b};\underline{a'})$. Then
%$$\dim W(\underline{b};\underline{a})\ge
%\dim W(\underline{b};\underline{a'})+\dim_kM_{\cB}(a_{t+c-2})_{0}-1-\ _0\!
%\hom_R(I_{Y},I_{X/Y}).$$
%\end{proposition}

\begin{proposition}\label{main1} \ Let $c\ge 3$, let $(X) \in
  W(\underline{b};\underline{a})$ and suppose \ $ \dim 
  W(\underline{b};\underline{a'}) \ge \lambda_{c-1}+ K_3 + K_4+...+K_{c-1} $
  and $\depth_{I(Z)}B\ge 2$ for a general $Y=\Proj(B)\in
  W(\underline{b};\underline{a'})$. If
\begin{equation}\label{maineq}
  _0\!
  \hom_R(I_{Y},I_{X/Y}) \le  \sum _{j=0}^{t+c-3} \binom{a_j-a_{t+c-2}+n}{n} \ ,
\end{equation}
then \ \ $\dim W(\underline{b};\underline{a})= \lambda_c+ K_3 +
K_4+...+K_c.$ We also get equality in \eqref{maineq}, as well as $$\dim
W(\underline{b};\underline{a}) = \dim
W(\underline{b};\underline{a'})+\dim_kM_{\cB}(a_{t+c-2})_{0}-1-\ _0\!
\hom_R(I_{Y},I_{X/Y}).$$
\end{proposition}
\begin{proof} Indeed the proof of Thm.\! 4.5 of \cite{KM} contains the ideas
  we need, but since the assumptions of Thm.\! 4.5 are
  different, we include a proof. %Looking backwards in the proof we also get
  %the equality in \eqref{maineq}.
  First we remark that we have $$ \lambda_c+ K_3 + K_4+...+K_c \ge \dim
  W(\underline{b};\underline{a})$$ by \cite{KM}, Prop.\!
  3.13 which combined with the assumption on $ \dim
  W(\underline{b};\underline{a'})$ yields
$$  \lambda _c-\lambda_{c-1}-K_c \ge \dim W(\underline{b};\underline{a})- \dim
W(\underline{b};\underline{a'}).$$ Next by \cite{KM}, Prop.\! 4.1 we have
the inequality 
$$\dim W(\underline{b};\underline{a}) -
\dim W(\underline{b};\underline{a'}) \ge \dim_kM_{\cB}(a_{t+c-2})_{0}-1-\ _0\!
\hom_R(I_{Y},I_{X/Y}).$$  Since $K_c=\ _0\! \hom(\coker
\varphi,R(a_{t+c-2}))$ by definition (see \cite{KM}, Prop.\! 3.12 and (3.14))
we can use \eqref{BR} and \eqref{Mi} to get
\begin{equation}\label{Mc0}
  \dim M_{\cB}(a_{t+c-2})_0-1=  \dim M_{\cA}(a_{t+c-2})_0=
\end{equation}
$$ \dim F^*(a_{t+c-2})_0-\dim
G^*(a_{t+c-2})_0+\ _0\! \hom(\coker \varphi,R(a_{t+c-2}))=$$
$$=\sum_{i=1}^t \binom{a_{t+c-2}-b_i+n}{n}- \sum
_{j=0}^{t+c-2} \binom{a_{t+c-2}-a_j+n}{n}+K_c.$$ Now looking at \eqref{lamda}
and noticing that $\lambda _{c-1}$ is defined by the analogous expression
where $a_j$ (resp $b_i$) ranges over $0\le j \le t+c-3$ (resp. $1\le i \le
t$), it follows after a straightforward computation that {\small
\begin{equation*} %\label{lambdas}
\lambda _c-\lambda_{c-1}=\sum_{i=1}^t
\binom{a_{t+c-2}-b_i+n}{n}- \sum _{j=0}^{t+c-2}
\binom{a_{t+c-2}-a_j+n}{n}-
 \sum _{j=0}^{t+c-3} \binom{a_j-a_{t+c-2}+n}{n}.
\end{equation*} }
Combining with \eqref{maineq}, we get $$\dim M_{\cB}(a_{t+c-2})_{0} -1- \ _0\!
\hom_R(I_{Y},I_{X/Y}) \ge  \lambda _c-\lambda_{c-1}+ K_c \ .$$
Hence all inequalities of displayed formulas in this proof are
equalities and we are done.
\end{proof}

\begin{theorem} \label{main} The Conjectures (and if $c>2$, the final
  dimension formula of Proposition~\ref{main1}) hold provided $$2 \le c \le 5
  \ \ {\rm and} \ \ n-c > 0 \ {\rm (\, supposing \ } char k = 0 {\rm \ if} \ c
  = 5\, )\, .$$
\end{theorem}
Indeed this is mainly \cite{KM}, Thm.\! 4.5, Cor.\! 4.7, Cor.\! 4.10, Cor.\!
4.14 and \cite{elli} ($c=2$) and \cite{KMMNP} ($c=3$). Moreover since the
proofs of \cite{KM} also show \eqref{maineq}, we get the final dimension
formula of Proposition~\ref{main1}. Moreover we have (valid also for $n=c$ and
$ char k \ne 0$):
%It is not quite easy to compute $ _0\! \hom_R(I_{Y},I_{X/Y})$ theoretically,
%but quite easy in concrete examples if we use Macaulay 2.

%Moreover from \cite{KM}, Rem.\! 4.16, 4.17, Cor.\! 4.18, see its proofs to get
%\eqref{maineq}, we have
\begin{remark} \label{dim0new} Assume $a_{0} > b_{t}$. Then \eqref{maineq} for
  $X$ general, and Conjecture~\ref{conj1} hold provided $3 \le c \le 5$
  (resp. $c>5$) and$ \ a_{t+c-2} > a_{t-2} $ (resp. $ \ a_{t+3} > a_{t-2}$)
  by \cite{KM09}, Thm\! 3.2.  
%$$(i_5): \quad
%a_{t+3}>a_{t-1}+a_t+a_{t+1}-a_0-a_1 \ \ \ {\rm provided} \ \ c=5 \ , $$ 
%$$(i_4):  \qquad  \qquad \qquad a_{t+2}>a_{t-1}+a_t-a_0  \ \ \ {\rm provided}
%\ \ c=4 \ , $$ 
% $$(i_3):  \qquad  \qquad \qquad  \qquad  \quad \quad a_{t+1}>a_{t-1} \ \  \ \
% {\rm provided} \ \ c=3 \ . $$
\end{remark}
%This is proved in \cite{KM}, Rem.\! 4.16, 4.17, Cor.\! 4.18. 
%%%%%%%%%%%%%%%%%%%%%%%%%%%%%%%%%%%%%%%%%%%%%%%%%%%%%%%%%%%%%%%%%%

\section{the codimension of the determinantal locus}

In this section we consider the problem of when the closure of $
W(\underline{b};\underline{a})$ is an irreducible component of $\Hi
(\PP^{n})$. If it is not a component, we determine its codimension in
$\Hi(\PP^{n})$ under certain assumptions. We also examine when $\Hi(\PP^{n})$
is generically smooth along $ W(\underline{b};\underline{a})$. Moreover we have
chosen to introduce the notion ``every deformation of $X$ comes from
deforming $\cA$'' because it gives the main reason for why $\overline{
  W(\underline{b};\underline{a})}$ is not always an irreducible component of
$\Hi (\PP^{n}).$

In the case the determinantal schemes are of dimension zero
or one, then $ \overline {W(\underline{b};\underline{a})}$ is not necessarily
an irreducible component of $\Hi(\PP^{n})$, as the following example shows.
\begin{example} [ $ \overline {W(\underline{b};\underline{a})}$ not an
  irreducible component in the range $0 \le n-c \le 1$, $c\ge 3$] \

  Let $\cB$ be a general $2 \times c$ matrix of linear entries %of $R=k[x_0,
  % x_1, \dots ,x_n]$
  and let $\cA =[\cB,v]$ where the entries of the column $v$ are general
  polynomials of the same degree $2$. The vanishing all $2 \times 2$ minors of
  $\cA$ defines a determinantal scheme $X$ of codimension $c$ in $\PP^n$.

  (i) Let $n=c$. Then $X=\Proj(A)$ is a reduced scheme of $2c+1$ points in
  $\PP^{c}$ and with $h$-vector $(\dim
  A_i)_{i=0}^{\infty}=(1,c+1,2c+1,2c+1,...)$. It follows that $\
  _vH^{1}_{\goth m}(A) \simeq H^1({\mathcal I}_X(v)) = 0$ for $v\ge 2$ and we
  get $\ _0\!\Hom_R (I_X,H^{1}_{\goth m}(A)) = 0$. By \eqref{Grad} the
  postulation Hilbert scheme %, $\Hi^H (\PP^{c})$
  is isomorphic to the usual Hilbert scheme at $(X)$, whose dimension is
  $c(2c+1)$. Moreover since the dimension of $ \overline
  {W(\underline{b};\underline{a})}$ is at most the conjectured value $
  c^2+4c-2$, and since
  $$ c^2+4c-2  < c(2c+1) \ \ {\rm for \  every \ }c \ge 3,  $$
  it follows that $ \overline {W({0,0};{1,1,...,1,2})}$ is not an irreducible
  component of $\Hi^H (\PP^{c})$.

  (ii) Let $n=c+1$. Then $X$ is a smooth connected curve in $\PP^{c+1}$ of
  degree $d=2c+1$ and genus $g=c$. Since the dimension of $ \overline
  {W(\underline{b};\underline{a})}$ is at most the conjectured value, which is
  $c^2+7c+2$, and since the dimension of the Hilbert scheme is at least
  $(n+1)d+(n-3)(1-g)=c^2+8c$, it follows that $ \overline
  {W({0,0};{1,1,...,1,2})}$ is not an irreducible component of $\Hi^p
  (\PP^{c+1})$ for every $c \ge 3$.
\end{example}

In what follows we briefly say ``$T$ a local ring'' (resp. ``$T$ artinian'')
for a local $k$-algebra $(T,{\mathfrak m}_T)$ essentially of finite type over
$k=T/{\mathfrak m}_T$ (resp. such that ${\mathfrak m}_T^r=0$ for some integer
$r$). Moreover we say ``$T \rightarrow S$ is a small artinian surjection''
provided there is a morphism $(T,{\mathfrak m}_T) \rightarrow (S, {\mathfrak
  m}_S)$ of local artinian $k$-algebras whose kernel ${\mathfrak a}$ satisfies
${\mathfrak a} \cdot {\mathfrak m}_T=0$.

Let $A=R/I_t(\cA)$. If $T$ is a local ring, we denote by $\cA_T=(f_{ij,T})$ a
matrix of homogeneous polynomials belonging to the graded polynomial algebra
$R_T:=R \otimes_k T$, satisfying $f_{ij,T} \otimes_T k=f_{ij}$ and $\deg
f_{ij,T} =a_j-b_i$ (if $\deg f_{ij,T}=0$ we let $ f_{ij,T}=0$) for all $i,j$.
Note that all elements from $T$ are considered to be of degree zero.

Once having such a matrix $\cA_T$, we get an induced morphism 
\begin{equation} \label{al} \varphi_T: F_T:=\oplus _{i=1}^tR_T(b_i)\rightarrow
  G_T:=\oplus_{j=0}^{t+c-2}R_T(a_j)
\end{equation}
 and we put $M_T= \coker \varphi_T^*$.
 \begin{lemma} \label{defalpha} If $X=\Proj(A)$, $A=R/I_t(\cA)$, is a standard
   determinantal scheme, then $A_T:=R_T/I_t(\cA_T)$ and $M_T$ are (flat)
   graded deformations of $A$ and $M$ respectively for every choice of
   $\cA_T$ as above. In particular $X_T=\Proj(A_T) \subset
   \PP^n_T:=\Proj(R_T) $ is a deformation of $X \subset \PP^n$ to $T$ with
   constant Hilbert function.
\end{lemma}
\begin{proof} Consulting \eqref{EN} and \eqref{BR} we see that the
  Eagon-Northcott and Buchsbaum-Rim complexes are functorial in the sense
  that, over $R_T$, all free modules and all morphisms in these
  complexes
  %all free modules in both the Eagon-Northcott and Buchsbaum-Rim complexes
  are induced by $ \varphi_T$, i.e. they are determined by $\cA_T$.
  Since these complexes become free resolutions of $A$ and $M$ respectively
  when we tensor with $k$ over $T$, it follows that $A_T$ and $M_T$ are flat
  over $T$ and satisfy $A_T \otimes_T k = A$ and $M_T \otimes_T k = M$.
\end{proof}
\begin{definition} \label{everydef} Let $X=\Proj(A)$, $A=R/I_t(\cA)$, be a
  standard determinantal scheme. We say ``every deformation of
  $X$ %(in $\Hi (\PP^{n})$)
  comes from deforming $\cA$'' if for every local ring $T$ and every graded
  deformation $R_T \to A_T$ of $R \to A$ to $T$, then $A_T$ is of the form
  $A_T=R_T/I_t(\cA_T)$ for some $\cA_T$ as above. Note that by \eqref{Grad} we
  can in this definition replace ``graded deformations of $ R \to A$'' by
  ``deformations of $X \hookrightarrow \PP^{n}$'' provided $\dim X \ge 1$.
\end{definition}
\begin{lemma} \label{unobst} Let $X=\Proj(A)$, $A=R/I_t(\cA)$, be a standard
  determinantal scheme, $(X) \in W(\underline{b};\underline{a})$. If every
  deformation of $X$ comes from deforming $\cA$, then $A$ (and hence $X$ if
  $\dim X \ge 1)$ is unobstructed. Moreover $\overline{
    W(\underline{b};\underline{a})}$ is an irreducible component of $\Hi
  (\PP^{n}).$
\end{lemma}
\begin{proof} Let $T \rightarrow S$ be a small artinian surjection and let
  $A_S$ be a deformation of $A$ to $S$. By assumption, $A_S=R_S/I_t(\cA_S)$
  for some matrix $\cA_S$. Since $T \rightarrow S$ is surjective, we can lift
  each $f_{ij,S}$ to a polynomial $f_{ij,T}$ with coefficients in $T$ such
  that $f_{ij,T} \otimes_T S=f_{ij,S}$. By Lemma~\ref{defalpha} it follows
  that $A_T:=R_T/I_t(\cA_T)$ is flat over $T$. Since $A_T \otimes_T S = A_S$
  we get the unobstructedness of $A$, as well as the unobstructedness of $X$
  in the case $\dim X \ge 1$ by \eqref{Grad}.

  % The argument above works for any local $k$-algebra $(T,{\mathfrak m}_T)$.
  % of
  % finite type over $k$.
  Finally let $T$ be the local ring of $\Hi (\PP^{n})$ at $(X)$ and let $A_T$,
  or $\Proj(A_T)$ if $\dim X \ge 1$, be the pullback of the universal object
  of $\Hi (\PP^{n})$ to $\Spec(T)$. Then there is a matrix $\cA_T =(f_{ij,T})$
  such that $A_T=R_T/I_t(\cA_T)$ by assumption. We can extend $f_{ij,T}$ to
  polynomials $f_{ij,D}$ with coefficients in $D$ where $\Spec(D) \subset \Hi
  (\PP^{n})$ is an open neighborhood of %$ \Hi (\PP^{n})$ containing
  $(X)$ for which the Eagon-Northcott complex associated to the matrix
  $\cA_D=(f_{ij,D})$ is exact at any $(X') \in \Spec(D)$ (cf. \cite{PS},
  Lem.\! 6.3 or \cite{elli}, proof of Thm.\! 1; in our case the existence of
  $\Spec(D)$ is quite easy since the Eagon-Northcott complex of the
  homogeneous coordinate ring of a standard
  determinantal scheme is always exact). It follows that $\Spec(D) \subset
  W(\underline{b};\underline{a})$, and since $\Spec(D)$ is open in $\Hi
  (\PP^{n})$ we are done.
 % By the universal property
 % of a representing object, there is a morphism $\Spec(D) \to \Hi (\PP^{n})$
 % with dense image around $(X)$ which factors via
 % $W(\underline{b};\underline{a})$ and we are done.
\end{proof}
\begin{remark} The arguments of these lemmas, which rely on the fact that we
  get $T$-flat schemes by just parameterizing the polynomials of 
%$\cA$ with coefficients from a local ring $T$, 
  $\cA$ over a local ring $T$, seem to be known, see e.g. Laksov's papers
  \cite{dan2}, \cite{dan} where he looks to flat families of determinantal
  schemes and their singular loci for arbitrary determinantal schemes. Indeed
  we may expect from Laksov's papers (or prove by other arguments, %together
  %with proving the unobstructedness of $A$ in the case every deformation comes
  %from deforming $\cA$, 
  as we remember Laksov did in a talk at the university of Oslo in the 70's)
  that families of {\it arbitrary} determinantal schemes obtained by
  parameterizing polynomials as above are $T$-flat; thus he mainly shows the
  unobstructedness part of Lemma~\ref{unobst}. Since we in this paper only
  look at determinantal schemes defined through {\it maximal} minors, our
  proofs are rather easy. Surprisingly enough the corresponding
  unobstructedness result of the {\it $R$-module} $M$ (of maximal grade
  \cite{eis}) seems less known. Indeed one may prove the unobstructedness of
  $M$ as we did
  for %show that modules of maximal grade are unobstructed by the
  %proof of the unobstructedness of
  $A$ in Lemma~\ref{unobst} because we from the Buchsbaum-Rim complex may see
  that every deformation of $M$ to $T$ comes from deforming $\cA$.
  %corresponding explicit dimension formula of $ \ _0\!\Ext^1_{R}(M,M)$, which
  %we will consider in a forthcoming paper, 
  We have learned, by distributing a preliminary version of this paper to
  specialists in deformations of modules, that the unobstructedness of $M$ was
  proved in Runar Ile's PhD thesis, cf. \cite{I01}, ch.\! 6 (Lem.\! 6.1.2 or
  Cor.\! 6.1.4).
\end{remark}
The following result is a key proposition to our work in this section. Here
the morphisms of the $\Ext^1$-groups are induced by the inclusion $I_{X/Y}
\hookrightarrow B$, e.g.
\begin{equation*} %\label{tau} 
 \tau_{X/Y} \ : \ _0\!
  \Ext^1_{B}(I_{Y}/I^2_{Y},I_{X/Y}) \rightarrow \ _0\! \Ext
  ^1_{B}(I_{Y}/I^2_{Y},B).
\end{equation*}
\begin{proposition} \label{mainTechn} Let $X=\Proj(A) \subset Y=\Proj(B)$ be
  good determinantal schemes defined by the vanishing of the maximal minors of
  ${\cA}$ and ${\cB}$ respectively where ${\cB}$ is
  obtained by deleting the last column of ${\cA}$. Let $Z \subset Y$ be a
  closed subset such that $U := Y-Z \hookrightarrow \PP^n$ is a local complete
  intersection (l.c.i.) and suppose \vskip 2mm

  {\rm (1)} $\depth_{I(Z)}B \geq 3 $, \ \ {\rm {\underline or}} 

  ${}$ \ \ \ \ $\depth_{I(Z)}B \geq 2 $ \ and \ $\rho^1: \
  _0\!\Ext^1_B(I_{X/Y},I_{X/Y})\to \ _0\!\Ext^1_B(I_{X/Y},B)$ is injective,

  {\rm (2)} $ \tau_{X/Y}: \ _0\! \Ext^1_{B}(I_{Y}/I^2_{Y},I_{X/Y})
  \rightarrow \ _0\! \Ext ^1_{B}(I_{Y}/I^2_{Y},B)$ is injective, and

  {\rm (3)} every deformation of \ $Y$ comes from deforming ${\cB}$. \\[2mm]
  Then every deformation of $X$ comes from deforming ${\cA}$.
  Moreover $$\dim_{(X)}\Hi (\PP^{n})=\dim_{(Y)}\Hi (\PP^{n})+\dim
  M_{\cB}(a_{t+c-2})_0-1-\ _0\! \hom_R(I_{Y},I_{X/Y}).$$
\end{proposition}

\begin{remark}\label{52NEW}
{\rm If $\depth_{I(Z)}B \ge 3$, then $\depth_{I(Z)}I_{X/Y} \ge 3$ and it
  follows from \eqref{NM} that
 $$_0\!\Ext^1_B(I_{X/Y},I_{X/Y}) \ =  \  _0\!\Ext^1_B(I_{X/Y},B) \ = \ 0 . $$ 
}
\end{remark}
\begin{remark}\label{HilbFlag} {\rm Let $\GradAlg(H_2,H_1)$ be the
    ``postulation'' Hilbert-flag scheme, i.e.  the representing object of the
    functor deforming surjections $(B \rightarrow A)$ of graded quotients of
    $R$ of positive depth at  $\goth m$, or equivalently flags $(X \subset Y)$
    of closed 
    subschemes of $ \PP^n$  with Hilbert functions $H_Y=H_B=H_2$ and $H_X=H_A=
    H_1$. In \cite{K04}, Prop.\! 4 (iii), we use theorems
    of Laudal on deformations of a category (\cite{L}) to show that the
    forgetful 
    morphism $$\GradAlg(H_2,H_1) \ra \GradAlg(H_1)$$ induced by $(X \subset Y)
    \ra (X)$, is} smooth {\rm and has fiber dimension $_0\!
    \hom_R(I_{Y},I_{X/Y})$ at $(X \subset Y)$ provided $B$ is unobstructed and
    (2) of Proposition~\ref{mainTechn} holds. By \eqref{Grad} this conclusion
    holds for the corresponding forgetful map from the Hilbert-flag scheme
    into the usual Hilbert scheme provided $X$ and $Y$ are ACM and $\dim X \ge
    1$.}
\end{remark}
\begin{proof} Let $R_T \to A_T$ be any graded deformation of $R \to A$ to a
  local ring $T$. By the smoothness of the forgetful map of
  Remark~\ref{HilbFlag}, there is a graded deformation $R_T \to B_T$ of $R \to
  B$ and a morphism $ B_T \to A_T$. By the assumption (3) there exists a
  matrix ${\cB_T}$ such that $B_T=R_T/I_t(\cB_T)$. By Lemma~\ref{defalpha}
  $\cB_T$ also defines a deformation $ M_{\cB_T}$ of $ M_{\cB}$.

  We will prove that there is a matrix $\cA_T$ such that $A_T=R_T/I_t(\cA_T)$
  and such that we get $\cB_T$ by deleting the last column of $\cA_T$. Looking
  at \eqref{Mi} and the text before and after \eqref{Mi}, we see that if we
  can find a section $\ B_T \to M_{\cB_T}(a_{t+c-2})$ which reduces to $\ B
  \to M_{\cB}(a_{t+c-2})$ via $(-) \otimes_T
  k$, %when we tensor by $k$ over $T$,
  we can use this section to define a column $v_T$ which allows us to put
  $\cA_T:=[\cB_T,v_T]$. % (which we may need to modify to get
  %$A_T=R_T/I_t(\cA_T)$). 
  Now since we have a
  deformation $ B_T \to A_T$ of $ B \to A$, it follows that
  $I_{X_T/Y_T}:=\ker(B_T \to A_T)$ is a deformation of $I_{X/Y} \simeq
  M_{\cB}(a_{t+c-2})^*$. If we sheafify and restrict to $U$ we get an
  isomorphism $\widetilde I_{X/Y}\arrowvert_U \simeq \widetilde
  M_{\cB}(a_{t+c-2})^* \arrowvert_U$ of invertible sheaves. Hence the flat
  sheaves $(\widetilde I_{X_T/Y_T})^*$ and $\widetilde M_{\cB_T}(a_{t+c-2})$
  are also isomorphic on the set $U_T$ which corresponds to $U$. Taking global
  sections, $H_{*}^{0}(U_T,-)$, of $\widetilde B_T \to (\widetilde
  I_{X_T/Y_T})^*$, we get a map which fits into a commutative diagram
 $$
 \xymatrix{\ \ B_T \ \ \ar@{->}[d] \ar[r] & \ \ \ar@{->}[d] H_{*}^{0}(U_T,
   \widetilde M_{\cB_T} (a_{t+c-2})) & \simeq \ \ \ \ M_{\cB_T}(a_{t+c-2}) \\ \
   \ B \ \ \ar[r] & \ \ H_{*}^{0}(U,\widetilde M_{\cB}(a_{t+c-2})) & \simeq \ \
   \ \ M_{\cB}(a_{t+c-2}) }
 $$
 where the lower isomorphism follows from the fact that $M_{\cB}$ is maximally
 CM, i.e. from $\depth_{I(Z)} M_{\cB} = \depth_{I(Z)}B \ge 2$. Note that since
 an $ M_{\cB}$-regular sequence lifts to an $ M_{\cB_T}$-regular sequence, we
 also have sufficient depth to get the upper isomorphism. Hence we get a
 section $\ B_T \to M_{\cB_T}(a_{t+c-2})$ and an induced matrix $\cA_T$, as
 required.

 Let $A'=R_T/I_t(\cA_T)$. We {\it claim} that $A'=A_T$, i.e. that $I_T'
 =I_{X_T/Y_T}$ where $I_T'=\ker(B_T \to A')$. Let $T_r:=T/{\mathfrak m}_T^r$,
 $B_{T_r}:=B_T \otimes _T {T_r}$, $I_{X_{T_r}/Y_{T_r}} := \ker(B_{T_r} \to A_T
 \otimes_T {T_r})$, $I_{T_r}'= \ker(B_{T_r} \to A' \otimes_T T_r)$ and
 $S:=T_{r-1}$. To prove the claim we first show that $ I_{T_r}'
 =I_{X_{T_r}/Y_{T_r}}$ for every integer $r > 0$. To see that this follows
 from the assumption (1), we suppose by induction that $I_S'= I_{X_S/Y_S}$ as
 ideals of $B_S$. Then $I_{T_r}'$ and $I_{X_{T_r}/Y_{T_r}}$ are two
 deformations of the same ideal $I_{X_S/Y_S} \hookrightarrow B_S$ to ${T_r}$
 and their difference, as {\it graded $B_{T_r}$ modules}, corresponds to an
 element, ${\rm diff}$, of $\ _0\!\Ext^1_B(I_{X/Y},I_{X/Y}) \otimes_k
 ({\mathfrak m}_T^{r-1} \cdot T_r)$ which via $\rho^1$ maps to a difference,
 $o_1-o_2 \in \ _0\!\Ext^1_B(I_{X/Y},B) \otimes_k ({\mathfrak m}_T^{r-1} \cdot
 T_r)$ where $o_i$ are the following obstructions. One of them, say $o_1$
 (resp. the other $o_2$) is the obstruction of deforming the {\it graded
   morphism} $I_S' \hookrightarrow B_S$ (i.e. the ideal) to a graded morphism
 between $I_T'$ and $B_T$ (resp. between $I_{X_{T_r}/Y_{T_r}}$ and $B_{T_r}$),
 cf. \cite{K04}, Rem.\! 3 for a similar situation. Since $I_T'$ and
 $I_{X_{T_r}/Y_{T_r}}$ are {\it ideals} in $B_{T_r}$, such graded morphisms
 exist. Hence $o_i=0$ for $i=1,2$, whence ${\rm diff}=0$ by the injectivity of
 $\rho^1$, and we conclude that $I_{T_r}' =I_{X_{T_r}/Y_{T_r}}$.

 To get the claim let $A'':=B_T/(I_T' + I_{X_T/Y_T})$. It suffices to show
 that the natural maps $A' \to A''$ and $A_T \to A''$ are isomorphisms. Note
 that $A'' \otimes_T k \simeq A'' \otimes_{R_T} R \simeq A$ and that we have
 similar isomorphisms for $A'$ and $A_T$. Notice also that every maximal ideal
 of $R_T$ lies over ${\mathfrak m}_T$. Hence we get both isomorphisms by the 
 % Nakayama's
 lemma of Nakayama, Azumaya and Krull if we can show that $A''$ is $T$-flat.
 But by the proof in the paragraph above the induced maps $A' \otimes_T T_r
 \to A'' \otimes_T T_r$ are isomorphism for every $r > 0$. It follows that
 $A'' \otimes_T T_r$ is $ T_r$-flat since $A' \otimes_T T_r$ is! Since $A''$
 is idealwise separated for ${\mathfrak m}_T$, we get that $A''$ is $T$-flat
 by the local criterion of flatness (see Thm.\! 20.C of \cite{Mat}).

 It remains to prove the dimension formula. Recall that there is a
 ``standard'' diagram (whose square is cartesian)
\begin{equation} \label{51bis}
  \begin{array}{ccccccccc}  & &  & & 0 \\
     & & & &  \downarrow \\  & & & & _0\!\Hom_R(I_{Y},I_{X/Y}) \\ &
    & & &   \downarrow \\ & & A^1 & \stackrel {T_{pr_1}}
    { \longrightarrow} & _0\!\Hom_R(I_{Y},B)
    \\  & &  \downarrow  &  \smallbox  & \downarrow \\ 
    _0\!\Hom (I_{X/Y},A) & \hookrightarrow & _0\!\Hom (I_{X},A) &
    \ra & _0\!\Hom_R(I_{Y},A) & \stackrel {\delta} { \longrightarrow} \ _0\!
    \Ext^1_{B}(I_{X/Y},A) \\ & & &
    & \downarrow  \\
    & &  & & 0
\end{array}
\end{equation}
which defines the tangent space $A^1$ of the Hilbert flag scheme
$\GradAlg(H_2,H_1)$ at $(B \to A)$ and where the morphisms $T_{pr_1}$ and
$\delta$ are natural maps (cf. \cite{K04}, (10) and note that the algebra
cohomology group $_0 {\rm H}^2(B,A,A) \simeq \ _0\!\Ext^1_{B}(I_{X/Y},A)$ (cf.
\cite{K04}, \S 1.1)). % and (5)). 
Under the assumption (2) the vertical sequence
is exact. We {\it claim} that $\delta = 0$. To see it, it suffices to prove
that $T_{pr_1}$ is surjective. The cartesian diagram is, however, well
understood in terms of the deformation theory of the Hilbert flag scheme.
Indeed if we take an arbitrary deformation $B_S$ of $B$ to the dual numbers
$S:=k[t]/(t^2)$, then $T_{pr_1}$ is surjective provided we can prove that
there is a deformation $(B_S \to A_S)$ of $(B \to A)$ to $S$. The latter
follows from the first part of the proof, or simply, from the assumption $(3)$
because we by (3) get $B_S=R_S/I_t(\cB_S)$ for some matrix $ \cB_S$ and we can
take $ \cA_S=[ \cB_S,v_S]$ where $v_S$ is any lifting of the last column of
$\cA$ to $S$. Letting $A_S:=R_S/I_t(\cA_S)$ we get the claim.

Since we have $\dim_{(X)}\Hi (\PP^{n})=\ _0\!\hom (I_{X},A)$ and
$\dim_{(Y)}\Hi (\PP^{n})=\ _0\!\hom (I_{Y},B)$ by Lemma~\ref{unobst}, we get
the dimension formula from the big diagram in which $\delta=0$ provided we can
prove that $\ _0\!\hom (I_{X/Y},A)=\dim M_{\cB}(a_{t+c-2})_0-1$. To see it we
apply $\Hom (I_{X/Y},-)$ onto $0 \to I_{X/Y} \to B \to A \to 0$. If we
use that $\Hom (I_{X/Y},B) \simeq M_{\cB}(a_{t+c-2})$, see
\eqref{DiMi}, we get the exact sequence
\begin{equation} \label{roseq}
0 \to B \to  M_{\cB}(a_{t+c-2}) \to  \Hom (I_{X/Y},A) \to 
\Ext^1_B(I_{X/Y},I_{X/Y})\to \Ext^1_B(I_{X/Y},B),
\end{equation} and we conclude by the assumption $(1)$.
\end{proof}
\begin{remark} \label{diamnoninj} Suppose $\tau_{X/Y}$ is not injective. Then
  the vertical sequence in the diagram \eqref{51bis} is not exact, and
  $\delta$ may be non-zero. It is, however, easy to enlarge the diagram
  \eqref{51bis} to a diagram of exact horizontal and vertical sequences by
  including $\ker \tau_{X/Y}$ and $\im \delta$ in the diagram. From this
  enlarged diagram it follows that
$$_0\!\hom (I_{X},A)= \ _0\!\hom (I_{X/Y},A) +   h^{0}({\cN}_Y)- \
_0\!\hom (I_{Y},I_{X/Y})+ \dim \ker \tau_{X/Y}- \dim \im \delta$$ since we
have $ \Hom (I_{Y},B) \simeq H^{0}_*(Y,{\cN}_Y)$ by \eqref{NM} and
$\depth_{\goth m} B \ge 2$. The displayed formula also holds if
$\tau_{X/Y}$ is injective.
% since in that case, $\delta =0$ by the proof above.
\end{remark}
\begin{theorem} \label{MainThm2} Suppose either $c=2$ and $n \ge 2$, \ {\rm
    {\underline or}} \  $3\le c \le 4$, $n-c \ge 2$ and $a_{i-\min (3,t)}\ge
  b_{i}$ for $\min (3,t)\le i \le t$. %i.e.$a_{i-3}\ge b_i$
 % for $3 \le i \le t$ if $t \ge 3$ and $a_0 \ge b_2$ if $t=2$. 
  If $W(\underline{b};\underline{a}) \ne \emptyset$, then $\overline{
    W(\underline{b};\underline{a})}$ is a generically smooth irreducible
  component of $\Hi (\PP^{n})$ of dimension $$\lambda_c + K_3+...+K_c \ .$$
  Moreover every deformation of a general $ (X) \in
  W(\underline{b};\underline{a})$ comes from deforming $\cA$.
\end{theorem}
For $c> 2$ this result is really \cite{KM}, Thm.\! 5.1 and Cor.\! 5.3 except
for the final statement. Since, however, the proof of Theorem~\ref{MainThm2}
is to apply Proposition~\ref{mainTechn} inductively to the flag \eqref{flag},
starting with the codimension 2 case where we by Lemma~\ref{HilbBu} know that
the final statement holds,
%every deformation of $X$ comes from deforming the Hilbert-Burch matrix
%(cf.\! \cite{elli}, proof of Thm.\! 2, \cite{PS}, proof of Thm.\! 6.2), 
the proof of \cite{KM} extends to get Theorem~\ref{MainThm2} for $c > 2$. If
$c=2$ we get the other conclusions (and even more) from \cite{elli} for $n >
2$ and from works of Gotzmann and others for $n=2$ as explained in \cite{K07},
Rem.\! 22 (i), or see \cite{KM}, Rem.\! 4.6 for a direct approach to $\dim
W(\underline{b};\underline{a})= \lambda_2$. (We also get all conclusions for
$c=2$ by combining Lemma~\ref{HilbBu} and Lemma~\ref{unobst}.)

\begin{lemma} \label{HilbBu} If $X=\Proj(A)$, $A=R/I_t(\cA)$, is a standard
  determinantal scheme of codimension $2$ in $\PP^{n}$ and $n \ge 2$, then
  every deformation of $X$ comes from deforming %the Hilbert-Burch matrix
  $\cA$.
\end{lemma}
\begin{proof}
  Let $\cA =(f_{ij})$ be a homogeneous $t \times (t+1)$ matrix which
  represents the morphism $\varphi^*$ of \eqref{gradedmorfismo} and let $R_T
  \to A_T$ be a graded deformation of $R \to A$ to a local ring $T$. To see
  that $A_T$ is of the form $A_T=R_T/I_t(\cA_T)$ for some matrix $\cA_T$
  reducing to $\cA$ via $(-) \otimes_T k$, we consider the canonical module
  $K_A=\Ext^2_R (A,R)(-n-1)$. Note that since $c=2$ we have
  $K_A(n+1-\ell_1)=M$, where $G^* \stackrel {\varphi^*} { \longrightarrow} F^*
  \to M \to 0$ is a part of the Buchsbaum-Rim complex, cf.\! \eqref{BR} and
  \eqref{ell}. Now we observe that $K_{A_T}:=\Ext^2_{R_T} (A_T,R_T)(-n-1)$ is a
  (flat) graded deformation of $K_A$ to $T$ because $\Ext^i_R (A,R) = 0$ for
  $i \ne 2$ (\cite{JS}, Proposition (A1)).
  % Since $K_{A_T}$ is a graded deformation of $K_A$, we get
  It follows that $K_{A_T}(n+1-\ell_1)=\coker (\varphi_T^*)$ where
  $\varphi_T^*$ corresponds to some matrix $\cA_T:=(f_{ij,T})$, as in
  \eqref{al}.

  Let $A':=R_T/I_t(\cA_T)$. It suffices to show that $A'$ and $A_T$ are
  isomorphic as $R_T$ quotients. Looking to the Eagon-Northcott complex
  associated to $A'$ over $R_T$ and dualizing, i.e. applying $
  \Hom_{R_T}(-,R_T)$ to it, we get back the part of the Buchsbaum-Rim complex
  where $\varphi_T^*$ appeared (up to twist). It follows that $K_{A_T} \simeq
  K_{A'}$ where $K_{A'}:=\Ext^2_{R_T} (A',R_T)(-n-1)$. Note that the
  Buchsbaum-Rim complex above is a free resolution of $K_{A'}(n+1-\ell_1)$ over
  $R_T$ since it is $T$-flat and reduces to a $R$-free resolution via a $(-)
  \otimes_T k$. Applying $ \Hom_{R_T}(-,R_T)$, we get $A' \simeq \Ext^2_{R_T}
  (K_{A'},R_T)(-n-1)$. Similarly by dualizing twice an $R_T$-free resolution
  of $A_T$ we show that $A_T \simeq \Ext^2_{R_T} (K_{A_T},R_T)(-n-1)$ and we are
  done.
  %(or the reader may instead want to get get same conclusion by showing $$A'
  %\simeq \Hom_{R_T}(K_{A'},K_{A'}) \simeq \Hom_{R_T}(K_{A_T},K_{A_T}) \simeq
  %A_T \ .$$
\end{proof}
\begin{remark} If we apply Proposition~\ref{mainTechn} successively to the
  flag \eqref{flag} it is straightforward to generalize Thm.\! 5.1 of
  \cite{KM} to the zero dimensional case, i.e. we may replace the condition $n
  \ge 1$ of \cite{KM}, Thm.\! 5.1 by $n \ge 0$ provided we in (i) of Thm.\!
  5.1 (and correspondingly in (ii) of Thm.\! 5.1) replace the $ \depth \ge 3$
  condition by the condition (1) of Proposition~\ref{mainTechn}.
\end{remark}
%
%\begin{theorem}\label{main5} Let $X\subset \PP^{n}$ be a general
 % determinantal scheme, let $X=X_c\subset X_{c-1}\subset
%  ...\subset X_2\subset \PP^{n}$ be the flag obtained by successively deleting
 % columns from the right hand side and 
%  let $D_i=R/I_{X_i}$ and $I_i=I_{X_{i+1}/X_{i}}$. If
 % $a_{i-\min (2,t)}\ge b_{i}$ for $\min (2,t)\le i \le t$, and if
%$$ \ _0\! \Ext
%^1_{D_i}(I_{X_i}/I^2_{X_i},I_i)=0 \mbox{ for } i=2,...,c-1 \ {\rm \ and } \
%_0\!\Ext^1_{D_{c-1}}(I_{c-1},I_{c-1})=0 \ \mbox{ if }\ c=n, $$ then $\overline{
%  W(\underline{b};\underline{a})}$ is a generically smooth irreducible
%component of $\Hi ^p(\PP^{n})$.
%\end{theorem}

There is a variation to Proposition~\ref{mainTechn} that we will use in the
case $n=c$ ($\dim X =0$) in which we mainly replace the injectivity assumption
in (1) for the $\Ext^1$-groups with the injectivity assumption for the
corresponding $\Ext^2$-groups. More precisely let
\begin{equation} \label{ro}
\rho^i: \ _0\!\Ext^i_B(I_{X/Y},I_{X/Y})\to \ _0\!\Ext^i_B(I_{X/Y},B).
\end{equation}
be the map induced by $I_{X/Y} \hookrightarrow B$. Then we have 
\begin{proposition} \label{varMainTechn} Let $X=\Proj(A) \subset Y=\Proj(B)$
  be good determinantal schemes defined by ${\cA}$ and ${\cB}$ where ${\cB}$
  is obtained by deleting the last column of ${\cA}$. Let $Z \subset Y$ be a
  closed subset such that $U := Y-Z \hookrightarrow \PP^n$ is a local complete
  intersection and suppose \vskip 2mm

  {\rm (1)} $_0\!\Ext^1_B(I_{X/Y},A) = 0$ (i.e.,
  $\rho^1$ surjective and $\rho^2$ injective) and $\depth_{I(Z)}B = 2 $,

  {\rm (2)} $\tau_{X/Y}: \ _0\! \Ext^1_{B}(I_{Y}/I^2_{Y},I_{X/Y})
  \hookrightarrow \ _0\! \Ext ^1_{B}(I_{Y}/I^2_{Y},B)$ is injective, and

  {\rm (3)} $Y$ is unobstructed (this is weaker than {\rm (3)} in
  Proposition~\ref{mainTechn}). \vskip 2mm

  Then $A$ is unobstructed and the postulation Hilbert scheme
  $\Hi^{H_A}(\PP^{n})$ satisfies
    $$\dim_{(X)}\Hi^{H_A}
    (\PP^{n})=\dim_{(Y)}\Hi^{p_Y} (\PP^{n})+\dim M_{\cB}(a_{t+c-2})_0-1-\ _0\!
    \hom_R(I_{Y},I_{X/Y}) + \dim \ker \rho^1.$$
\end{proposition}
\begin{proof} Let $T \rightarrow S$ be a small artinian surjection with kernel
  ${\mathfrak a}$, and let $R_S \to A_S$ be any graded deformation of $R \to
  A$ to $S$. By Remark~\ref{HilbFlag}, there is a graded deformation $R_S \to
  B_S$ of $R \to B$ and a morphism $ B_S \to A_S$. By assumption (3) and
  \eqref{Grad} there exists a deformation $R_T \to B_T$ of $R_S \to B_S$ to
  $T$. It is well known that the algebra cohomology group $_0 {\rm H}^2(B,A,A)
  \otimes_k {\mathfrak a}$ contains the obstruction of deforming $ B_S \to
  A_S$ further to $B_T$ and that there is an injection $_0 {\rm H}^2(B,A,A)
  \hookrightarrow \ _0\!\Ext^1_{B}(I_{X/Y},A)$ (\cite{SGA7}, exp.\! VI,
  \cite{K04}, \S 1.1). The rightmost group vanishes by (1), and it follows
  that $A$ is unobstructed.

  Finally we get the dimension formula from \eqref{51bis}. Indeed the
  arguments are almost exactly the same as in the proof of
  Proposition~\ref{mainTechn} with the variation that \eqref{roseq} now
  implies $$\ _0\!\hom (I_{X/Y},A)=\dim M_{\cB}(a_{t+c-2})_0-1 + \dim \ker
  \rho^1.$$
\end{proof}
\begin{remark} %Let $B \rightarrow A$ be a graded morphism of quotients of $R$.
  We say that ``$A$ is unobstructed along any graded deformation of $B$''
  (call this phrase (*)) if for every small artinian surjection $T \rightarrow
  S$ and for {\it every} graded deformation $B_S \rightarrow A_S$ of $B
  \rightarrow A$ to $S$, there exists, for {\it every} graded deformation
  $B_T$ of $B_S$ to $T$, a graded deformation $B_T \rightarrow A_T$ reducing
  to $B_S \rightarrow A_S$ via $(-) \otimes_T S$. It is clear from the proof
  above that $ \ _0\!\Ext^1_{B}(I_{X/Y},A)=0$ implies (*) and moreover that we
  can generalize Proposition~\ref{varMainTechn} by replacing the assumption $
  \ _0\!\Ext^1_{B}(I_{X/Y},A)=0$ by (*).
\end{remark}
Now we come to the main results of this paper which are direct consequences of
the above propositions. We start with determinantal curves whose result we
will need in the zero dimensional case. Note that the result below is known
(\cite{KMMNP}, Cor.\! 10.15 and Rem.\! 10.9 for $c=3$, \cite{KM}, Rem.\! 5.4
and Cor.\! 5.7 for $4 \le c \le 5$) except for the final statement of (i) and
most statements on the codimension in (ii) and (iii). With notations as in
Proposition~\ref{mainTechn} we have
\begin{proposition} \label{Dim1Prop} Let $X=\Proj(A)$, $A=R/I_t(\cA)$, be
  general in $ W(\underline{b};\underline{a})$ and suppose $a_{i-\min (3,t)}\ge
  b_{i}$ for $\min (3,t)\le i \le t$, $\dim X =n-c=1$ and $3 \le c \le 5$ (and
  $char k =0$ if $c=5$). If $Y=\Proj(B)$ is defined by the vanishing of the
  maximal minors of ${\cB}$ where ${\cB}$ is obtained by deleting the last
  column of ${\cA}$, then the following statements are true:

  \vskip 2mm {\rm (i)} If $\tau_{X/Y}:\, _0\!
  \Ext^1_{B}(I_{Y}/I^2_{Y},I_{X/Y}) \rightarrow \! _0\! \Ext
  ^1_{B}(I_{Y}/I^2_{Y},B)$ is injective, then $X$ is unobstructed and
  $\overline{ W(\underline{b};\underline{a})}$ is a generically smooth
  irreducible component of $\Hi ^p(\PP^{n})$ of dimension $\lambda_c +
  K_3+...+K_c$.
% e.g. $\codim_{\Hi ^p(\PP^{n})}\overline{W(\underline{b};\underline{a})}=0$. 
  Moreover every deformation of $X$ comes from deforming ${\cA}$.

  {\rm (ii)} If $_0\! \Ext^1_{A}(I_{X}/I^2_{X},A)=0$, then $X$ is
  unobstructed, $\dim W(\underline{b};\underline{a})= \lambda_c +
  K_3+...+K_c$ and $$\codim_{\Hi ^p(\PP^{n})}\overline{
    W(\underline{b};\underline{a})}=\dim \ker \tau_{X/Y} - \ _0\! \ext
  ^1_{B}(I_{X/Y},A) \ .$$

  {\rm (iii)} We always have \ $\dim W(\underline{b};\underline{a})= \lambda_c
  + K_3+...+K_c$ and $$\codim_{\Hi ^p(\PP^{n})} \overline{
    W(\underline{b};\underline{a})} \le \dim \ker \tau_{X/Y} \ .$$ Moreover if
  $\ _0\! \Ext ^1_{B}(I_{X/Y},A)=0$, then we have equality in the codimension
  formula if and only if $X$ is unobstructed.
 \end{proposition}
 \begin{proof} In all cases we use Theorem~\ref{main} to get $\dim {
     W(\underline{b};\underline{a})}= \lambda_c + K_3+...+K_c$.

   (i) Since Theorem~\ref{MainThm2} applies to $
   W(\underline{b};\underline{a'}) \ni (Y)$ where $\underline{a'} = a_0,
   a_1,..., a_{t+c-3}$, we get (i) from Proposition~\ref{mainTechn},
   Remark~\ref{dep} and Lemma~\ref{unobst}.

   (ii) The vanishing of the $\Ext^1$-group of (ii) implies that $X$ is
   unobstructed ($X$ is l.c.i. by Remark~\ref{dep}), and moreover that $\im
   \delta \simeq \ _0\! \Ext ^1_{B}(I_{X/Y},A)$, cf.\! the diagram
   \eqref{51bis} and continue the horizontal exact sequence as a long exact
   sequence of algebra cohomology. Since we have $h^0(\cN_Y)- \dim
   W(\underline{b};\underline{a'})=0$ by Theorem~\ref{MainThm2} and $\
   _0\!\hom (I_{X/Y},A)=\dim M_{\cB}(a_{t+c-2})_0-1$ by \eqref{roseq} and
   Remark~\ref{52NEW}, we conclude by Remark~\ref{diamnoninj} and the final
   dimension formula of Theorem~\ref{main}.

   (iii) As in (ii) we get $ \ _0\!\hom (I_{X},A)- \dim
   W(\underline{b};\underline{a})= \dim \ker \tau_{X/Y}- \im \delta$ and hence
   the inequality of (iii). If the $\ _0\! \Ext ^1_{B}(I_{X/Y},A)$ vanishes,
   then $\im \delta = 0$, and since one knows that $X$ is unobstructed if and
   only if we have equality in $h^0(\cN_X) \ge \dim_{(X)} {\Hi ^p(\PP^{n})}$,
   we conclude easily. 
   % by using $h^0(\cN_Y)- \dim W(\underline{b};\underline{a'})=0$, $\
   % _0\!\hom (I_{X/Y},A)=\dim M_{\cB}(a_{t+c-2})_0-1$,
   % Remark~\ref{diamnoninj} and Theorem~\ref{main}.
 \end{proof}
 \begin{remark} [for the case where the codimension of $X$ in $\PP^{n}$ is 3,
   i.e. $c=3$]
   \ \label{va}

   (i) Since $Y$ is licci (\cite{KMMNP}, Def.\! 2.10) for $c=3$ , we always
   have $ \ _0\! \Ext ^1_{B}(I_{Y}/I^2_{Y},B) =0$ by \cite{B} (or see
   \cite{her} or \cite{KMMNP}, Prop.\! 6.17) and hence we get $\ker \tau_{X/Y}
   \simeq \ _0\! \Ext^1_{B}(I_{Y}/I^2_{Y},I_{X/Y})$.

   (ii) It is shown in \cite{KMMNP}, Cor.\! 10.11 (for $n-c=1$) and Cor.\!
   10.17 (for $n-c=0$) that $ \ _0\! \Ext ^1_{B}(I_{Y}/I^2_{Y},I_{X/Y}) =0$
   provided $a_{t+1} > a_{t} +a_{t-1}-b_1$. Indeed the proofs of \cite{KMMNP}
   (or \cite{KM}, Cor.\! 5.10 (i)) show $ \ _0\! \Ext ^1_{R}(I_{Y},I_{X/Y})
   =0$ by mainly using the $R$-free minimal resolution of $I_{Y}$ and the
   degree of the minimal generators of $I_{X/Y}$ which we get from \eqref{EN}.
   Hence we can conclude by the injection $ \ _0\! \Ext
   ^1_{B}(I_{Y}/I^2_{Y},I_{X/Y}) \hookrightarrow \ _0\! \Ext
   ^1_{R}(I_{Y},I_{X/Y})$. The vanishing of $ \ _0\! \Ext
   ^1_{B}(I_{Y}/I^2_{Y},I_{X/Y})$ is, however, much more common than given by
   the above argument. Indeed examining many examples by Macaulay 2
   (\cite{Mac}) in the range $a_0 > b_t=b_1$ we almost always got $ \ _0\!
   \Ext ^1_{B}(I_{Y}/I^2_{Y},I_{X/Y})=0$ provided $a_{t+1} > 3+b_{t}-n+c$ and
   $0 \le n-c \le 1$.
 %have so far always got $ \ _0\! \Ext ^1_{B}(I_{Y}/I^2_{Y},I_{X/Y})=0$ in
 %  the case $a_{t+1} > 2+b_{t-1}$ for $n-c=1$ and $a_{t+1} > 3+b_{t-1}$ for
 %  $n-c=0$.
\end{remark}
\begin{example}  [determinantal curves in $\PP^{4}$, i.e. with $c =
  3$] \label{excu}  \ \label{va}

   (i) Let $\cB$ be a general $2 \times 3$ matrix of linear entries and let
   $\cA= [\cB,v]$ where the coordinates of the column $v$ are general
   polynomials of the same degree $m$, $m>0$. The vanishing of all $2 \times
   2$ minors defines a smooth curve $X={X_m}$ of degree $3m+1$ and genus
   $3m(m-1)/2$ in $\PP^{4}$. %We claim that $\ _0\! \Ext
   %^1_{B}(I_{{X_m}/Y},A)=0$ for every $m>0$. Indeed since
   %$I_{{X_m}/Y}=K_B^{*}(-m-2)$ by \eqref{Di} and \eqref{ell}, it follows that
   %$\ _0\! \Ext ^2_{B}(I_{{X_m}/Y},I_{{X_m}/Y} )$ is independent of $m$. Using
   %Macaulay 2 we compute its dimension for $m=1$ only and we get $0$. Since $\
   %_0\! \Ext ^1_{B}(I_{{X_m}/Y},A) \hookrightarrow \ _0\! \Ext
   %^2_{B}(I_{{X_m}/Y},I_{{X_m}/Y} )$ is injective for every $m$ by
   %Remark~\ref{52NEW}, we get the claim.
   By Macaulay 2 (mainly),
\begin{equation*}  \label{flagvanish}
      _0\! \Ext^1_{B}(I_{Y}/I^2_{Y},I_{{X_m}/Y})=0 \ \ {\rm if
       \ and \ only \ if } \ \ \ m \ne 2.
   \end{equation*} Its dimension is 1 if $m=2$ in which
   case $ \ _0\! \Ext^1_{A}(I_{{X_m}}/I^2_{{X_m}},A)=0$ and $\ _0\! \Ext
   ^1_{B}(I_{{X_m}/Y},A)=0$. Note that we above  % in \eqref{flagvanish}
   only  need to use  Macaulay 2 for $m \le 2$ because the
   condition  $a_{t+1} > a_{t} +a_{t-1}-b_1$ of Remark~\ref{va} (ii) is
   equivalent to 
  % the injection $\ _0\! 
 % \Ext^1_{B}(I_{Y}/I^2_{Y},I_{{X_m}/Y}) \ \hookrightarrow \ _0\!
 %  \Ext^1_{R}(I_{Y},I_{{X_m}/Y})$ and the resolution $0 \to R(-3)^2 \to
 %  R(-2)^3 \to I_{Y} \to 0$ show \eqref{flagvanish}  for 
   $m >2$. It follows from Proposition~\ref{Dim1Prop} (i) that $\overline{
     W(\underline{b};\underline{a})}$ is a generically smooth irreducible
   component of $\Hi ^p(\PP^{4})$ of dimension $\lambda_3 + K_3$ for $m \ne 2$,
   and from either (ii) or (iii) that ${X_m}$ is unobstructed and $\codim_{\Hi
     (\PP^{4})}\overline{ W(\underline{b};\underline{a})}= \ _0\!
   \ext^1_{B}(I_{Y}/I^2_{Y},I_{{X_m}/Y})= 1 $ for $m=2$. Hence $\dim_{({X_m})}
   {\Hi (\PP^{4})} = \lambda_3 + K_3+1$ in this case.

   Finally computing $\lambda_3$ and $K_3$ by their definitions, we get
   $\lambda_3 + K_3 = 17+(m+1)(3m+4)/2$ for $m >1$ and $21$ for $m=1$.

   (ii) Let $\cA$ be a general $2 \times 4$ matrix whose columns consist of
   general polynomials of the same degree, $1$, $2$, $3$ and $m$, $m \ge 3$
   respectively. Put $\cA= [\cB,v]$ where the coordinates of the column $v$
   are all of degree $m$. The vanishing of all $2 \times 2$ minors of $\cA$
   defines a smooth curve $X=:{X_m}$ of degree $11m+6$ and genus
   $(11m^2+29m+8)/2$ in $\PP^{4}$. By Remark~\ref{va} we get $\ _0\!
   \Ext^1_{B}(I_{Y}/I^2_{Y},I_{{X_m}/Y})=0$ for $m > 5$, but a Macaulay 2
   computation shows this vanishing also for $3 \le m \le 5$. It follows from
   Proposition~\ref{Dim1Prop} (i) that $\overline{
     W(\underline{b};\underline{a})}$ is a generically smooth irreducible
   component of $\Hi ^p(\PP^{4})$ of dimension $\lambda_3 + K_3 =$ $$
   85+m(11m-5)/2 \ \ {\rm \ for} \ m >3 \ ,  \ \ {\rm \ and} \  \ \ 126 \ {\rm \
     for} \ m=3.$$

   We can also analyze the cases $m=1$ and $2$ by using
   Proposition~\ref{Dim1Prop} (i). Note that we now delete the column of
   degree 3 polynomials to define $\cB$, i.e. we let $Y$ be defined by the
   maximal minors of the $2 \times 3$ matrix $\cB$ consisting of linear (resp.
   degree $m$) entries in the first and second (resp. third) column. If $m=1$
   (resp. $m=2$) one verifies that $ _0\!
   \Ext^1_{B}(I_{Y}/I^2_{Y},I_{{X_m}/Y})=0$ by Macaulay 2 and we get that
   $\overline{ W(\underline{b};\underline{a})}$ is a generically smooth
   irreducible component of $\Hi ^H(\PP^{3})$ of dimension $66$ (resp. $96$).
\end{example}
\begin{remark}
  We have checked the vanishing of $ \ _0\!
  \Ext^1_{A}(I_{{X_m}}/I^2_{{X_m}},A)$ for several $m \ge 1$ in
  Example~\ref{excu} (ii). It seems that this group is always non-zero for
  every $m \ge 1$. This, we think, shows that the results presented here are
  quite strong because it is hard to show unobstructedness and to find
  $\dim_{(X_m)}{\Hi ^H(\PP^{n})}$ when even the ``smallest known obstruction
  group, $ \ _0\! \Ext^1_{A}(I_{{X_m}}/I^2_{{X_m}},A),$'' does not vanish.
\end{remark}
Now we consider zero dimensional determinantal schemes ($n-c=0$). Indeed
Proposition~\ref{mainTechn} with $\depth_{I(Z)}B= 2 $ and
Proposition~\ref{varMainTechn} are designed to take care of this case. We
restrict our attention to a general $X$ which through Remark~\ref{dep} imply
that all depth conditions of the propositions are satisfied. Then our result
leads e.g. to the unobstructedness of $A$ where $X=\Proj(A)$. In fact for
special choices of $X$, $A$ may be obstructed \cite{siq}. First we consider
codimension $c=3$ determinantal subschemes and schemes with $c \ge 4$ which we
may treat similarly.
 \begin{theorem} \label{Dim0Thm} Let $X=\Proj(A)$, $A=R/I_t(\cA)$, be general
   in $ W(\underline{b};\underline{a})$ and let $Y=\Proj(B)$ and $
   V=\Proj(C)$  be defined by the vanishing of the
   maximal minors of ${\cB}$ and ${\cC}$ respectively where ${\cB}$ (resp.
   ${\cC}$) is obtained by deleting the last column of ${\cA}$ (resp.
   ${\cB}$). Suppose $\dim X =n-c=0$, $a_{i-3}\ge b_{i}$ for $\min (3,t)\le i
   \le t$ and suppose that \eqref{maineq} holds. Moreover suppose $$ {\rm \
     either \ \ } c=3 {\rm \ \ or \ \ [ } \ 4 \le c \le 6 \ \ {\rm and \ }
   \ker \tau_{Y/V} = 0 {\rm \ ] } , $$ and suppose $char k = 0$ if $c=6$. Then
   the following statements are true:

   \vskip 2mm {\rm (i)} If both $\rho^1: \, _0\!\Ext^1_B(I_{X/Y},I_{X/Y})\to \,
   _0\!\Ext^1_B(I_{X/Y},B)$ and $\tau_{X/Y}:\, _0\!
   \Ext^1_{B}(I_{Y}/I^2_{Y},I_{X/Y}) \rightarrow \,_0\! \Ext
   ^1_{B}(I_{Y}/I^2_{Y},B)$ are injective, then $A$ is unobstructed and
   $\overline{ W(\underline{b};\underline{a})}$ is a generically smooth
   irreducible component of the postulation Hilbert scheme $\Hi ^H(\PP^{c})$
   of dimension $\lambda_c + K_3+...+K_c$. Moreover every deformation of $X$
   comes from deforming ${\cA}$.
%.g.\ $\codim_{\Hi ^p(\PP^{n})}\overline{ W(\underline{b};\underline{a})}=0$.

{\rm (ii)} If $\ _0\! \Ext ^1_{B}(I_{X/Y},A)=0$ and \ $ \ker \tau_{X/Y}= 0$,
then $\overline{ W(\underline{b};\underline{a})}$ belongs to a unique
generically smooth irreducible component $Q$ of $\Hi ^H(\PP^{c})$ of
codimension $\dim \ker \rho^1$. Indeed $A$ is unobstructed and $$\dim Q
=\lambda_c + K_3+...+K_c + \dim \ker \rho^1 . $$

{\rm (iii)} If \ $_0\! \Ext^1_{A}(I_{X}/I^2_{X},A)=0$, then $A$ is
unobstructed, $\dim W(\underline{b};\underline{a})= \lambda_c + K_3+...+K_c$
and $$\codim_{\Hi ^H(\PP^{c})}\overline{ W(\underline{b};\underline{a})}=\dim
\ker \rho^1 + \dim \ker \tau_{X/Y}- \ _0\! \ext ^1_{B}(I_{X/Y},A) .$$

{\rm (iv)} We always have \ $\codim_{\Hi ^H(\PP^{c})}\overline{
  W(\underline{b};\underline{a})} \le \dim \ker \rho^1 +  \dim \ker
\tau_{X/Y}$. \\ 
Suppose $\ _0\! \Ext ^1_{B}(I_{X/Y},A)=0$. Then we have $$\codim_{\Hi
  ^H(\PP^{c})}\overline{ W(\underline{b};\underline{a})} = \dim \ker \rho^1+
\dim \ker \tau_{X/Y} \ $$ if and only if $A$ is unobstructed.
 \end{theorem}
 \begin{proof} In all cases we use Proposition~\ref{main1} to get $\dim {
     W(\underline{b};\underline{a})}= \lambda_c + K_3+...+K_c$ since the
   Conjectures hold for $W(\underline{b};\underline{a'}) \ni (Y)$ where
   $\underline{a'} = a_0, a_1,..., a_{t+c-3}$ by Theorem~\ref{main}. Moreover
   if we apply Proposition~\ref{Dim1Prop} (i) to $Y \subset V \subset
   \PP^{n}$, $(Y) \in W(\underline{b};\underline{a'})$ (provided $c > 3$, if
   $c=3$ we apply Theorem~\ref{MainThm2} to $W(\underline{b};\underline{a'})
   \ni (Y)$ ), it follows that {\it every deformation of\, $Y$ comes from
     deforming ${\cB}$}.
  % Lemma~\ref{unobst} applies to $(Y) \in W(\underline{b};\underline{a'})$.
  
   (i) Using the above statements we easily conclude by
   Proposition~\ref{mainTechn} and Lemma~\ref{unobst}.

   (ii) Now we use Proposition~\ref{varMainTechn} instead of
   Proposition~\ref{mainTechn}. Comparing the dimension formula of Proposition
   ~\ref{varMainTechn} with the final one of Proposition~\ref{main1} and using
   that $\dim_{(Y)} {\Hi ^p(\PP^{c})}= \dim W(\underline{b};\underline{a'})$
   by Lemma~\ref{unobst}, we get all conclusions of (ii).

   (iii) The vanishing of the $\Ext^1$-group implies that $A$ is unobstructed
   and that $\im \delta \simeq \ _0\! \Ext ^1_{B}(I_{X/Y},A)$, cf.\!
   \eqref{51bis}. We have $\ _0\!\hom (I_{X/Y},A)=\dim M_{\cB}(a_{t+c-2})_0-1
   + \dim \ker \rho^1$ by \eqref{roseq} and $h^0(\cN_Y)- \dim
   W(\underline{b};\underline{a'})=0$ by Lemma~\ref{unobst}. We conclude by
   Remark~\ref{diamnoninj} and the final dimension formula of
   Proposition~\ref{main1}. %, we
   % conclude by Remark~\ref{diamnoninj}, \eqref{roseq} and the final
   % dimension
   % formula of Proposition~\ref{main1}.

   (iv) The proof is similar to the last part of (iii), cf.\! the proof of
   Proposition~\ref{Dim1Prop} (iv).
 \end{proof}

 \begin{remark} (i) Looking to the proofs we see that we don't need to suppose
   \eqref{maineq} to get the conclusions of (i) and (ii)
   which don't involve dimension and codimension formulas. 

   (ii) Note the overlap in (ii) and (iv) of the theorem.
\end{remark}
In \cite{KMMNP}, Ex.\! 10.18 we considered the example $\cA= [\cB,v]$ where
$\cB$ was a general $2 \times 3$ matrix of linear entries and where the
coordinates of the column $v$ are general polynomials of the same degree $m$,
$m>2$. Using Macaulay 2 one may easily check that $\ _0\! \Ext
^1_{B}(I_{{X}/Y},I_{{X}/Y})=0$ for $m=3$. Since $I_{{X}/Y}=K_B^{*}(-m-1)$ by
\eqref{Di} and \eqref{ell}, it follows that $\ _0\! \Ext
^1_{B}(I_{{X}/Y},I_{{X}/Y} )$ is independent of $m$, and hence vanishes for
every $m \ge 3$. The families of zero dimensional schemes given in
\cite{KMMNP}, Ex.\! 10.18 are therefore generically smooth of known dimension
by Theorem~\ref{Dim0Thm} (i). More advanced examples are given in the examples
below where several aspects of Theorem~\ref{Dim0Thm} are used. Note that the
condition \eqref{maineq} in Theorem~\ref{Dim0Thm} is taken care of by
Remark~\ref{dim0new}.

\begin{example} [Using Theorem~\ref{Dim0Thm} (i), mainly, for zero
  dimensional schemes in $\PP^{3}$] \  \label{thmi}

  Let $\cA= [\cB,v]$ be a general $2 \times 4$ matrix with linear (resp.
  cubic) entries in the first and second (resp. third) column and where the
  entries of $v$ are polynomials of the same degree $m$, $m \ge 3$. The
  vanishing of all $2 \times 2$ minors defines a reduced scheme $X=:{X_m}$ of
  $7m+3$ points in $\PP^{3}$. One verifies that $ _v\!
  \Ext^1_{B}(I_{Y}/I^2_{Y},I_{{X_m}/Y})=0$, $v \le 0$, and $\ _0\! \Ext
  ^1_{B}(I_{{X_m}/Y},I_{{X_m}/Y} )=0$ for $m=3$ by Macaulay 2 mainly and since
  $I_{{X_m}/Y}=K_B^{*}(-m+1)$ by \eqref{Di} and \eqref{ell}, we get the same
  conclusion for every $m \ge 3$. It follows from Theorem~\ref{Dim0Thm} (i)
  that $\overline{ W(\underline{b};\underline{a})}$ is a generically smooth
  irreducible component of $\Hi ^H(\PP^{3})$ of dimension $\lambda_3 +
  K_3=7m+25$ (resp. 45) for $m > 3$ (resp. $m = 3$).

  If $m=1$ or $2$ we delete the column of degree 3 polynomials to define
  $\cB$, i.e. we let $Y$ be defined by the maximal minors of the $2 \times 3$
  matrix $\cB$ consisting of linear (resp. degree $m$) entries in the first
  and second (resp. third) column. If $m=1$ one verifies (by Macaulay 2) that
  $ _0\! \Ext^1_{B}(I_{Y}/I^2_{Y},I_{{X_m}/Y})=\ _0\! \Ext
  ^1_{B}(I_{{X_m}/Y},I_{{X_m}/Y} )=0$ and we get by Theorem~\ref{Dim0Thm} (i)
  that $\overline{ W(\underline{b};\underline{a})}$ is a generically smooth
  irreducible component of $\Hi ^H(\PP^{3})$ of dimension $\lambda_3 +
  K_3=22$. If $m=2$ one verifies that $ _0\!
  \Ext^1_{B}(I_{Y}/I^2_{Y},I_{{X_m}/Y})= 0$ and that $\ _0\! \ext
  ^1_{B}(I_{{X_m}/Y},I_{{X_m}/Y} )=2$. Hence we can not use
  Theorem~\ref{Dim0Thm} (i), but we can use Theorem~\ref{Dim0Thm} (ii)! Such
  cases are more thoroughly explained in the next example. We % need to verify
  verify that $ _0\! \Ext^1_{B}(I_{{X_m}/Y},B)=0$, to get $\dim \ker \rho^1 =
  \ _0\! \ext^1_{B}(I_{{X_m}/Y},I_{{X_m}/Y})$, and that $ _0\!
  \Ext^1_{B}(I_{{X_m}/Y},A)=0$. %It holds (Macaulay 2).
  We conclude that $\overline{ W(\underline{b};\underline{a})}$ is contained
  in a generically smooth irreducible component of $\Hi ^H(\PP^{3})$ of
  dimension $\lambda_3 + K_3 +\ _0\! \ext ^1_{B}(I_{{X_m}/Y},I_{{X_m}/Y} )=
  37$.
\end{example}
  
\begin{example} [Using Theorem~\ref{Dim0Thm} (ii), mainly, for
  zero dimensional schemes in $\PP^{3}$] \  \label{thmii}

  Similar to Example~\ref{excu} (ii) we let $\cA = [\cB,v]$ be a general $2
  \times 4$ matrix whose columns consist of general polynomials of the same
  degree, $1$, $2$, $3$ and $m$, $m \ge 3$ respectively. The vanishing of all
  $2 \times 2$ minors of $\cA$ defines a reduced scheme $X=:{X_m}$ of $11m+6$
  points in $\PP^{3}$. This time Macaulay 2 computations show $ _0\!
  \ext^1_{B}(I_{{X_m}/Y},I_{{X_m}/Y})=2$ and $ _v\!
  \Ext^1_{B}(I_{{X_m}/Y},B)=0$ for every $m \ge 3$ and every $v \ge 0$ (we
  only need to check it for $m=3$ because $I_{{X_m}/Y}=K_B^{*}(-m+2)$ by
  \eqref{Di} and \eqref{ell}). It follows that $$\dim \ker \rho^1 = \ _0\!
  \ext^1_{B}(I_{{X_m}/Y},I_{{X_m}/Y})=2$$ for every $m \ge 3$, i.e. we can not
  use Theorem~\ref{Dim0Thm} (i) at all. We have, however, $\ _0\!
  \Ext^1_{B}(I_{Y}/I^2_{Y},I_{{X_m}/Y})=0$ for $m > 5$ by Remark~\ref{va} (ii)
  and $\ _0\! \ext^1_{B}(I_{Y}/I^2_{Y},I_{{X_m}/Y})=2$ (resp. $0$) for $m=3$
  (resp. $3 < m \le 5$) by Macaulay 2. Since $ _0\!
  \Ext^2_{B}(I_{{X_m}/Y},I_{{X_m}/Y})=0$ for $m=3$ and hence for every $m \ge
  3$, we get $ _0\! \Ext^1_{B}(I_{{X_m}/Y},A)=0$ for $m \ge 3$. We can
  therefore apply Theorem~\ref{Dim0Thm} (ii) in this situation except when
  $m=3$. In the latter case
   % (but not in the case $4\le m \le 6$)
   $ \ _0\! \Ext^1_{A}(I_{{X_m}}/I^2_{{X_m}},A)=0$ and Theorem~\ref{Dim0Thm}
   (iii) applies. Hence Theorem~\ref{Dim0Thm} applies for every $m \ge 3$, and
   we get that $\overline{ W(\underline{b};\underline{a})}$ belongs to a
   unique generically smooth irreducible component of $\Hi ^H(\PP^{3})$ of
   codimension $2$ (resp. $4$) for $m>3$ (resp. $ m = 3$). Indeed $A$ is
   unobstructed and $$\dim \overline{ W(\underline{b};\underline{a})}
   =\lambda_3 + K_3= 11m+35 \ \ {\rm \ for} \ m >3 \ \ \ {\rm \ and} \ \ \ 67
   \ {\rm \ for} \ m=3.$$ We remark that we have checked a possible vanishing
   of $ \ _0\! \Ext^1_{A}(I_{{X_m}}/I^2_{{X_m}},A)$ for several $m \ge 3$, and
   in the range $3 < m \le 6$ this group is non-zero. 

   Finally to be complete we consider the cases $m=1$ and $m=2$ in which case
   we will delete the column of degree $3$ polynomials to define $\cB$ and
   hence $Y$. If $m=1$ we get by Macaulay 2 $\dim \ker \rho^1 = \ _0\!
   \ext^1_{B}(I_{{X_m}/Y},I_{{X_m}/Y})=2$, $\ _0\!
   \Ext^1_{B}(I_{Y}/I^2_{Y},I_{{X_m}/Y})=0$ and $ _0\!
   \Ext^1_{B}(I_{{X_m}/Y},A)=0$. We have $\dim \overline{
     W(\underline{b};\underline{a})} =\lambda_3 + K_3= 35 \ $
   and $$\codim_{\Hi ^H(\PP^{3})}\overline{
     W(\underline{b};\underline{a})}=\dim \ker \rho^1 = 2 $$ by
   Theorem~\ref{Dim0Thm} (ii). Moreover if $m=2$ we get $\dim \ker \rho^1 = \
   _0\! \ext^1_{B}(I_{{X_m}/Y},I_{{X_m}/Y})=4$, $\ _0\!
   \ext^1_{B}(I_{Y}/I^2_{Y},I_{{X_m}/Y})=1$, $ _0\!
   \Ext^1_{B}(I_{{X_m}/Y},A)=0$ and $ \ _0\!
   \Ext^1_{A}(I_{{X_m}}/I^2_{{X_m}},A)=0$ by Macaulay 2. By
   Theorem~\ref{Dim0Thm} (iii) we find $\dim  {
     W(\underline{b};\underline{a})} =\lambda_3 + K_3= 53$ \ and $$\codim_{\Hi
     ^H(\PP^{3})}\overline{ W(\underline{b};\underline{a})}=\dim \ker \rho^1 +
   \dim \ker \tau_{X/Y} = 4+1=5 .$$
 \end{example}

 \begin{example} [Using Theorem~\ref{Dim0Thm} (i) with non-vanishing
   obstruction groups] \

   We let $\cA = [\cB,v]$ be a general $2 \times 4$ matrix whose columns
   consist of general polynomials of the same degree, $2$, $2$, $4$ and $m$,
   $m \ge 4$ respectively. The vanishing of all $2 \times 2$ minors of $\cA$
   defines a reduced scheme $X=:{X_m}$ of $20m+16$ points in $\PP^{3}$. This
   time Macaulay 2 computations show $ _0\!
   \ext^1_{B}(I_{{X_m}/Y},I_{{X_m}/Y})=0$ and $\ _0\!
   \Ext^1_{B}(I_{Y}/I^2_{Y},I_{{X_m}/Y})=0$ (resp. $=1$) for every $m > 4$
   (resp. $m=4$). It follows from Theorem~\ref{Dim0Thm} (i) that $\overline{
     W(\underline{b};\underline{a})}$ is a generically smooth irreducible
   component of $\Hi ^H(\PP^{3})$ of dimension $\lambda_3 + K_3=20m+49$ for $m
   > 4$. For $m = 4$ we have verified that $ \ _0\!
   \ext^1_{A}(I_{{X_m}}/I^2_{{X_m}},A)=3$ and in this particular case we have
   not been able to verify whether $A$ is unobstructed or not. But for every
   $m > 4$, $A$ is unobstructed by Theorem~\ref{Dim0Thm} (i)! Moreover we have
   checked a possible vanishing of $ \ _0\!
   \Ext^1_{A}(I_{{X_m}}/I^2_{{X_m}},A)$ for many $m \ge 3,$ and combined with
   some theoretical arguments (which we don't take here) we can conclude that
   this group is always non-zero for every $m \ge 4$. Again, we think, this
   shows that the results presented here are quite strong because it is really
   hard to show unobstructedness %and to find $\dim_{(X_m)}{\Hi ^H(\PP^{n})}$
   when even the ``smallest known obstruction group, $ \ _0\!
   \Ext^1_{A}(I_{{X_m}}/I^2_{{X_m}},A),$'' does not vanish.
 \end{example}

 In the final case $ 4 \le c \le 6$ and $\ker \tau_{Y/V} \ne 0$ where a
 general $X=\Proj(A) \subset Y=\Proj(B) \subset V=\Proj(C)$ is given by
 deleting columns as above we can not apply Proposition~\ref{mainTechn} to $X
 \subset Y$ because there is no reason to expect condition (3) of
 Proposition~\ref{mainTechn} to be true (that condition is closely related to
 $\ker \tau_{Y/V} = 0$). But we can still use Proposition~\ref{varMainTechn}
 since condition (3) of Proposition~\ref{varMainTechn} is weakened to "$Y$
 unobstructed''. %(and we can replace $ 4 \le c \le 6$ provided \eqref{maineq}
 % holds for $ W(\underline{b};\underline{a'})$). 
The natural condition for "$Y$ unobstructed'' which also give a formula for
$h^0(\cN_Y)- \dim W(\underline{b};\underline{a'} )$ is $_0\!
\Ext^1_{B}(I_{Y}/I^2_{Y},B)=0$, cf. the proof of Proposition~\ref{Dim1Prop}
(ii). We get
 %That part does not require $c \le 5$. Hence we get

 \begin{proposition} \label{Dim0Prop} With notations as above, suppose $ 4 \le
   c \le 6$ (let $char k = 0$ if $c=6$), $\dim X =n-c=0$, $a_{i-3}\ge b_{i}$
   for $\min (3,t)\le i \le t$ and suppose that \eqref{maineq} holds. Then $\dim
   { W(\underline{b};\underline{a})}= \lambda_c + K_3+...+K_c$ and the
   following statements are true:

   {\rm (i)} If $\ _0\! \Ext ^1_{B}(I_{X/Y},A)=0$, $_0\!
   \Ext^1_{B}(I_{Y}/I^2_{Y},I_{X/Y})=0$ and \ $_0\!
   \Ext^1_{B}(I_{Y}/I^2_{Y},B)=0$ then $A$ is unobstructed. Moreover $
   W(\underline{b};\underline{a})$ is contained in a unique generically smooth
   irreducible component of $\Hi ^H(\PP^{c})$ of codimension $\dim \ker \rho^1
   + \dim \ker \tau_{Y/V}- \ _0\! \ext ^1_{C}(I_{Y/V},B)$.

%{\rm (ii)} If \ $_0\! \Ext^1_{A}(I_{X}/I^2_{X},A)=0$ and \ {\rm [ }$ _0\! \Ext
%^1_{C}(I_{Y/V},B)=0$ {\underline or} $_0\! \Ext^1_{B}(I_{Y}/I^2_{Y},B)=0$ {\rm
%  ] }, then $A$ is unobstructed. Moreover $
%  W(\underline{b};\underline{a})$ belongs to a unique generically smooth
%irreducible component of $\Hi ^H(\PP^{c})$ of codimension $$\dim \ker \rho^1 +
%\dim \ker \tau_{X/Y}- \ _0\! \ext ^1_{B}(I_{X/Y},A)+\dim \ker \tau_{Y/V} - \
%_0\! \ext ^1_{C}(I_{Y/V},B) \ .$$ MAYBE WE NEED TO RESTRIVT TO $c < 6$

   {\rm (ii)} We always have $\codim_{\Hi ^H(\PP^{c})}\overline{
     W(\underline{b};\underline{a})} \le \dim \ker \rho^1 + \dim \ker
   \tau_{X/Y}+\dim \ker \tau_{Y/V}$.  \\
   Suppose $\ _0\! \Ext ^1_{B}(I_{X/Y},A)=0$ and $\ _0\! \Ext
   ^1_{C}(I_{Y/V},B)=0$. Then we have $$\codim_{\Hi ^H(\PP^{c})}\overline{
     W(\underline{b};\underline{a})} = \dim \ker \rho^1+ \dim \ker
   \tau_{X/Y}+\dim \ker \tau_{Y/V}$$ if and only if $A$ is unobstructed (e.g.
   $_0\! \Ext^1_{A}(I_{X}/I^2_{X},A)=0$).
 \end{proposition}
 \begin{proof} We have $\dim { W(\underline{b};\underline{a})}= \lambda_c +
   K_3+...+K_c$ by Proposition~\ref{main1} since Theorem~\ref{main} applies to
   $W(\underline{b};\underline{a'})$ where $\underline{a'} = a_0,
   a_1,..., a_{t+c-3}$. 

   (i) This follows from Proposition~\ref{varMainTechn} and
   Proposition~\ref{Dim1Prop} (ii) and by comparing the dimension formula of
   Proposition~\ref{varMainTechn} with
   the final one of Proposition~\ref{main1}.

  % (ii) The vanishing of $_0\! \Ext^1_{A}(I_{X}/I^2_{X},A)$ implies that $A$
  % is unobstructed and that $\im \delta \simeq \ _0\! \Ext ^1_{B}(I_{X/Y},A)$,
  % cf.\! \eqref{51bis}. We have $\ _0\!\hom (I_{X/Y},A)=\dim
  % M_{\cB}(a_{t+c-2})_0-1 + \dim \ker \rho^1$ by \eqref{roseq} and
  % $h^0(\cN_Y)- \dim W(\underline{b};\underline{a'})=+ \dim \ker \tau_{Y/V}- \
  % _0\! \ext ^1_{C}(I_{Y/V},B)$ by Proposition~\ref{Dim1Prop} (ii) provided
  % $_0\! \Ext^1_{B}(I_{Y}/I^2_{Y},B)=0$. We conclude by
  % Remark~\ref{diamnoninj} and the final dimension formula of
  % Proposition~\ref{main1}. 
 
   (ii) Combining Remark~\ref{diamnoninj} and the final formula of
   Proposition~\ref{main1} with \eqref{roseq},
   % $\ _0\!\hom (I_{X/Y},A)=\dim M_{\cB}(a_{t+c-2})_0-1 + \dim \ker \rho^1$,
   we get $$ \ _0\!\hom (I_{X},A)- \dim W(\underline{b};\underline{a})=
   h^0(\cN_Y)- \dim W(\underline{b};\underline{a'}) + \dim \ker \rho^1+ \dim
   \ker \tau_{X/Y}- \im \delta.$$ By the same argument we have $ h^0(\cN_Y)-
   \dim W(\underline{b};\underline{a'}) \le \dim \ker \tau_{Y/V}$ and
   moreover, if $\ _0\! \Ext ^1_{C}(I_{Y/V},B)=0$, then equality holds. Hence
   we get the inequality of (ii), and furthermore, if the two $\Ext^1$-groups
   of (ii) vanish then the inequality of (ii) is an equality if and only if
   $\dim_{(X)} {\Hi ^H(\PP^{c})}= \ _0\!\hom (I_{X},A)$ and we are done.
 \end{proof}
\begin{example} [determinantal zero dimensional schemes in $\PP^{4}$, i.e.
  with $c = 4$] \

  Let $\cA=[\cB,v]$ be a general $2 \times 5$ matrix with linear (resp.
  quadratic) entries in the first, second and third (resp. fourth) column and
  let both entries of the column $v$ be of degree $m \ge 2$. Keeping the
  notations of Proposition~\ref{Dim0Prop}, we get that the vanishing of all $2
  \times 2$ minors defines a reduced scheme $X$ of $7m+2$ points in $\PP^{4}$.
  One verifies that $\ \dim \ker \rho^1 = \ _0\! \ext
  ^1_{B}(I_{{X}/Y},I_{{X}/Y} )=3$, $\ _0\! \Ext ^1_{B}(I_{X/Y},A)=0$ and that
  \eqref{maineq} holds % for $2\le m \le 3$ by Macaulay 2 (cf.
  by Remark~\ref{dim0new}. % for $m>3$). 
Note that we have $\dim \ker \tau_{Y/V} =
  1$ and $\ _0\! \Ext ^1_{C}(I_{Y/V},B)=0$ from Example~\ref{excu} (i).

  Suppose $m>2$. Then $\ _0\! \Ext^1_{B}(I_{Y}/I^2_{Y},I_{{X}/Y})= \ _0\!
  \Ext^1_{B}(I_{Y}/I^2_{Y},B)=0$ for every $m>2$ and it follows from
  Proposition~\ref{Dim0Prop} (i) that $A$ is unobstructed and $\dim  {
    W(\underline{b};\underline{a})}= \lambda_4 + K_3 + K_4 = 7m+31$. Hence
  $W(\underline{b};\underline{a})$ is contained in a unique generically smooth
  irreducible component of the postulation Hilbert scheme $\Hi ^H(\PP^{4})$
  and,
  $$\codim_{\Hi ^H(\PP^{4})}\overline{ W(\underline{b};\underline{a})} =
  \dim \ker \rho^1+ \dim \ker \tau_{Y/V}=3+1=4.$$

  Suppose $m=2$. Since $ \dim \ker \tau_{X/Y}=\ _0\!
  \ext^1_{B}(I_{Y}/I^2_{Y},I_{{X}/Y})=4$ and $_0\!
  \Ext^1_{A}(I_{X}/I^2_{X},A)=0$, it follows from Proposition~\ref{Dim0Prop}
  (ii) that $A$ is unobstructed and that $\dim {
    W(\underline{b};\underline{a})}= \lambda_4 + K_3 + K_4 = 44$. Hence
  $W(\underline{b};\underline{a})$ is contained in a unique generically smooth
  irreducible component of $\Hi ^H(\PP^{4})$ and,
  $$\codim_{\Hi ^H(\PP^{4})}\overline{ W(\underline{b};\underline{a})} =
  \dim \ker \rho^1+ \dim \ker \tau_{X/Y}+\dim \ker \tau_{Y/V}=3+4+1=8.$$ In
  this case we see that all three kernels of Proposition~\ref{Dim0Prop} (ii)
  contribute to the codimension of $ W(\underline{b};\underline{a})$ in $\Hi
    ^H(\PP^{4})$!

\end{example}
\begin{remark} \label{correctKM} If we apply Proposition~\ref{mainTechn}
  successively to the flag \eqref{flag} we get Prop.\! 10.12 and Thm.\! 10.13
  of \cite{KMMNP} in a correct version (the injectivity of $\rho^1$, i.e. the
  assumption (1) of Proposition~\ref{mainTechn} in the case $\depth_{I(Z)}B =
  2 $ lacked in \cite{KMMNP}). Indeed in \cite{KM}, Rem.\! 6.3 we announced
  that some results in \S 10 of \cite{KMMNP} were inaccurate, and in the new
  hypothesis (*) of Rem.\! 6.3 we increased the depth assumption of the
  corresponding hypothesis in \cite{KMMNP} by 1 to get valid results. The new
  hypothesis (*) applies to determinantal schemes of positive dimension, i.e.
  the results of \cite{KMMNP}, \S 10 hold in this case. In the zero
  dimensional case we introduced, in addition to (*) of Rem.\! 6.3, an
  assumption (Rem.\! 6.3 (ii)), which is equivalent to the injectivity of
  $\rho^1$. This
  assumption %, together with $X_{c-1} \hookrightarrow \PP^n$ an l.c.i.,
  makes the results of \cite{KMMNP}, \S 10 correct in the zero dimensional case.
  In \cite{KM}, Rem.\! 6.3 (ii) we indicate a proof for this claim, and now
  Proposition~\ref{mainTechn} provides us with another proof. In \cite{KM},
  Rem.\! 6.3 (i) and (iii), we claimed that e.g. the unobstructedness of $A$
  also implied {\it all} results of \cite{KMMNP}, \S 10, but this is a little
  inaccurate because the very final result of \cite{KMMNP} (Cor.\! 10.17) uses
  the injectivity of $\rho^1$ to get the dimension formula. E.g. in
  Example~\ref{thmii} for $m>3$ (resp. Example~\ref{thmi} with $m=2$) the
  formula of Cor.\! 10.17 gives $\dim_{(X)} {\Hi ^H(\PP^{3})}= \dim  {
    W(\underline{b};\underline{a})}$, which should be correct according to
  Rem.\! 6.3 (i) (resp. Rem.\! 6.3 (iii)). The correct dimension is, however,
  $\dim_{(X)} {\Hi^H(\PP^{3})}= \dim  {
    W(\underline{b};\underline{a})} + \dim \ker \rho^1$, $\dim \ker \rho^1=2$
  in both cases. This observation is a reason for writing this paper, namely
  to provide detailed proofs in the zero dimensional case for the correction
  ``Rem.\! 6.3 (ii)'' %and $X_{c-1} \hookrightarrow \PP^n$ an l.c.i.''
  and to present several results related to \cite{KM}, Rem.\! 6.3 (i) and
  (iii) (see Proposition~\ref{varMainTechn}, Theorem~\ref{Dim0Thm},
  Proposition~\ref{Dim0Prop} where we see that we have to add $\dim \ker
  \rho^1$ to get valid (co)dimension formulas. Note also the obvious misprint
  in Rem.\! 6.3, that $I_c$ should have been $I_{c-1}$).
  % since the claim in Rem.\! 6.3 (i) and (iii) are not true as stated.
  Thus, letting $D_i=R/I_{X_i}$ and $I_i=I_{X_{i+1}/X_{i}}$, the following
  hypothesis makes {\it all} results of \cite{KMMNP}, \S 10, true for good
  determinantal schemes $X$ with $\dim X \ge 0$;

  {\em Given $X\subset \PP^{n}$ a good determinantal scheme of dimension
    $n-c$, we will assume that there exists a flag $X=X_c\subset
    X_{c-1}\subset ...\subset X_2\subset \PP^{n}$ such that for each $i<c$,
    the closed embedding $X_{i+1}\hookrightarrow X_i$ is l.c.i. outside some
    set $Z_i$ of codimension 2 in $X_{i+1}$ ($\depth _{Z_i}\cO _{X_{i+1}}\ge
    2$). Moreover, we suppose $X_2\hookrightarrow \PP^{n}$ is an l.c.i. in
    codimension $\le 1$ and if \ $ c=n$ we suppose that $\
    _0\!\Ext^1_{D_{c-1}}(I_{c-1},I_{c-1}) \hookrightarrow \
    _0\!\Ext^1_{D_{c-1}}(I_{c-1},D_{c-1})$ is injective.}
\end{remark}
  
%%%%%%%%%%%%%%%%%%%%%%%%%%%%%%%%%%%%%%%%%%%%%%%%%%%%%%%%%%%%%%%%%%

\end{document}